\documentclass[12pt]{article}
\usepackage{tikz}
\usepackage{amsmath}
\usepackage{amssymb}
\usepackage{amsthm}
\usepackage{fullpage}
\usepackage{verbatim}
\usepackage{stmaryrd}
\usepackage{authblk}
\usepackage{multirow}
\usepackage{mathrsfs}

\usepackage[colorlinks]{hyperref}
\usepackage{mathdots}
\usepackage{graphicx,subfigure}
\usepackage[english]{babel}
\usepackage[utf8]{inputenc}
\usepackage{tikz}
\usepackage{mathtools}

\theoremstyle{definition}
\newtheorem{Theorem}{Theorem}
\newtheorem{Corollary}{Corollary}
\newtheorem{Lemma}[equation]{Lemma}
\newtheorem{Proposition}{Proposition}

\theoremstyle{definition}
\newtheorem{Definition}[equation]{Definition}
\newtheorem{Definition-Remark}[equation]{Definition/Remark}
\newtheorem{Example}[equation]{Example}

\newtheorem{Convention}[equation]{Convention}

\theoremstyle{remark}

\newtheorem{Remark}[equation]{Remark}

\numberwithin{equation}{section}
\numberwithin{figure}{section}

\newcommand{\C}{\mathbb{C}}
\newcommand{\N}{\mathbb{N}}
\newcommand{\R}{\mathbb{R}}

\newcommand{\Z}{\mathbb{Z}}

\newcommand{\mc}[1]{\mathcal{#1}} 
\newcommand{\mbf}[1]{\mathbf{#1}} 
\newcommand{\mt}[1]{\text{#1}}
\newcommand{\BM}[1]{\overline{\text{B}}_{#1}}

 \def \bb{ \atopwithdelims \llbracket \rrbracket}

\begin{document}

\title{Stirling Posets}

\author[1]{Mahir Bilen Can}
\author[2]{Yonah Cherniavsky}

\affil[1]{{\small Tulane University, New Orleans; mahirbilencan@gmail.com}}    
\affil[2]{{\small Ariel University, Israel; yonahch@ariel.ac.il}}

\normalsize

\date{\today}
\maketitle

\begin{abstract}
We define combinatorially a partial order on the set partitions 
and show that it is equivalent to the Bruhat-Chevalley-Renner
order on the upper triangular matrices. 
By considering subposets consisting of set partitions with 
a fixed number of blocks, we introduce and investigate 
``Stirling posets.'' As we show, the Stirling posets have a hierarchy 
and they glue together to give the whole set partition poset.
Moreover, we show that they (Stirling posets) are 
graded and EL-shellable. 
We offer various reformulations of their length functions
and determine the recurrences for their length generating series.
\vspace{.2cm}

\end{abstract}

\section{Introduction}

Let $n$ be a nonnegative integer. 
A collection $S_1,\dots, S_r$ of subsets 
of an $n$-element set $S$ is said to 
be a set partition of $S$ if $S_i$'s ($i=1,\dots, r$) are mutually disjoint
and $\cup_{i=1}^r S_i = S$. In this case,
$S_i$'s are called the blocks of the partition.
If $n>0$ and $S=\{1,\dots, n\}$, the collection of all 
set partitions of $S$ is denoted by $\Pi_n$. 
We will often drop set parentheses and commas and just put 
vertical bars between blocks. 
If $B_1,\dots, B_k$ are the blocks of a set 
partition $\pi$ from $\Pi_n$, then the {\em standard form}
of $\pi$ is defined as $B_1|B_2|\cdots |B_k$, where 
we assume that $\min B_1 <\cdots < \min B_k$
and the elements of each block are listed in increasing order. 
For example, 
$\pi = 136 | 2459 | 78$ is a set partition from $\Pi_9$. 

The set $\Pi_n$ is known to be a host to many interesting 
algebraic and combinatorial structures. Among 
these structures is the following well studied partial ordering: let $A$
and $A'$ be two set partitions of $S$.
$A$ is said to {\em refine} $A'$ if each block of $A$ 
is contained in some block of $A'$. 
This ``refinement ordering'' makes 
$\Pi_n$ into a lattice,
called the partition lattice, 
and by a result of Pudlak and Tuma (see~\cite{PudlakTuma})
it is known that every lattice is isomorphic to a sublattice of $\Pi_n$
for some $n$. 

A property that is shared
by all partition lattices is that their order complexes have the homotopy
type of a wedge of spheres. This important combinatorial topological property is 
seen by analyzing the labelings of the covering relations 
of the refinement ordering. Indeed, it follows as a consequence
of the fact that the refinement ordering is an ``edge lexicographically 
shellable'' (EL-shellable for short) poset as shown by Gessel 
(mentioned in~\cite{Bjorner80}) and by Wachs in~\cite{Wachs}. We postpone the proper
definition of EL-shellability to our preliminaries section but 
let us only mention very briefly that the property of EL-shellability
of a graded poset is a way of linearly ordering of  
the maximal faces of the associated order complex,
say $F_1,\cdots,F_m$, in such a way that $F_k \cap \left( \cup_{i=1}^{k-1} F_i \right)$
is a nonempty union of maximal proper faces of $F_k$ ($k=2,\dots, m$).
Having this property immediately implies a plethora of results 
on the topology of the underlying poset, such as Cohen-Macaulayness. 
It is also helpful for better understanding the M\"obius function 
of the poset. 
Our purpose in this paper is to present 
another natural partial ordering on $\Pi_n$ 
and to show that our poset is EL-shellable as well. 
To define our ordering we start with defining 
its most basic ingredient, namely the ``arc-diagram.''
It is customary to call a linearly ordered poset a chain. 
Here we will identify chains by their Hasse diagrams and 
draw them in an unorthodox way, horizontally, by placing the smallest 
entry on the left and connecting the vertices by arcs.
In Figure~\ref{F:chain1} we depicted the chain on 9 vertices,
where each arc represents a covering relation.
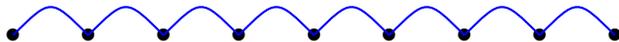
\begin{figure}[h]
\begin{center}
\begin{tikzpicture}[scale=.5]
\node at (-8,0) {$\bullet$};
\node at (-6,0) {$\bullet$};
\node at (-4,0) {$\bullet$};
\node at (-2,0) {$\bullet$};
\node at (0,0) {$\bullet$};
\node at (2,0) {$\bullet$};
\node at (4,0) {$\bullet$};
\node at (6,0) {$\bullet$};
\node at (8,0) {$\bullet$};
\draw[thick,blue,-] (-8,0) ..controls  (-7,1) .. (-6,0);
\draw[thick,blue,-] (-6,0) ..controls  (-5,1) .. (-4,0);
\draw[thick,blue,-] (-4,0) ..controls  (-3,1) .. (-2,0);
\draw[thick,blue,-] (-2,0) ..controls  (-1,1) .. (0,0);
\draw[thick,blue,-] (0,0) ..controls  (1,1) .. (2,0);
\draw[thick,blue,-] (2,0) ..controls  (3,1) .. (4,0);
\draw[thick,blue,-] (4,0) ..controls  (5,1) .. (6,0);
\draw[thick,blue,-] (6,0) ..controls  (7,1) .. (8,0);
\end{tikzpicture}
\caption{A chain on 9 vertices.}
\label{F:chain1}
\end{center}
\end{figure}

\begin{Definition}
By a {\em labeled chain} we mean a chain
whose vertices are labeled by distinct 
numbers. 
An {\em arc-diagram on $n$ vertices} is a disjoint union of 
labeled chains where the labels are from $\{1,\dots,n\}$
and each label $i\in \{1,\dots, n\}$ is used exactly once. 
\end{Definition}
See Figure~\ref{F:introexample} for an example. 
\begin{figure}[h]
\begin{center}
\begin{tikzpicture}[scale=.5]
\node at (-8,0) {$\bullet$};
\node at (-6,0) {$\bullet$};
\node at (-4,0) {$\bullet$};
\node at (-2,0) {$\bullet$};
\node at (0,0) {$\bullet$};
\node at (2,0) {$\bullet$};
\node at (4,0) {$\bullet$};
\node at (6,0) {$\bullet$};
\node at (8,0) {$\bullet$};
\node at (-8,-0.5) {$1$};
\node at (-6,-0.5) {$2$};
\node at (-4,-0.5) {$3$};
\node at (-2,-0.5) {$4$};
\node at (0,-0.5) {$5$};
\node at (2,-0.5) {$6$};
\node at (4,-0.5) {$7$};
\node at (6,-0.5) {$8$};
\node at (8,-0.5) {$9$};
\draw[thick,blue,-] (-8,0) ..controls  (-0.5,3) .. (6,0);
\draw[thick,blue,-] (-6,0) ..controls  (-3,1) .. (0,0);
\draw[thick,blue,-] (-4,0) ..controls  (0,2) .. (4,0);
\draw[thick,blue,-] (0,0) ..controls  (1,1) .. (2,0);
\draw[thick,blue,-] (2,0) ..controls  (5,1) .. (8,0);
\end{tikzpicture}
\caption{An arc-diagram on 9 vertices}
\label{F:introexample}
\end{center}
\end{figure}
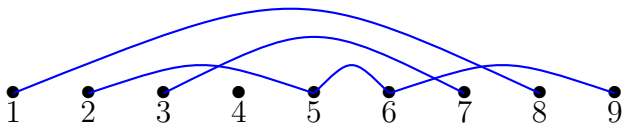

It is easy to see that the arc-diagrams 
on $n$ vertices are in bijection with the 
elements of $\Pi_n$. Indeed, the map 
that is defined by grouping the labels 
of a chain into a set extends to define
a bijection from arc-diagrams to the 
set partitions. 
For example, under this bijection, 
the arc-diagram in Figure~\ref{F:introexample} 
corresponds to the set partition $18|2569|37|4$ 
in $\Pi_9$. In the light of this bijection,
from now on, we will work with 
the arc-diagrams instead of set partitions.
Let us use the notation $\mc{A}_n$ for denoting 
the set of all arc-diagrams on $n$ vertices. 
The goal of our article is to endow $\mc{A}_n$ 
with a partial order and to use it to investigate 
certain subposets of $\mc{A}_n$. 
In particular, we will focus on the subposets
$\mc{A}_{n,k}\subset \mc{A}_n$, where the elements of 
$\mc{A}_{n,k}$ have exactly $k$ chains. 
We will call these subposets as the title of our paper,
namely, the Stirling posets.

Next we proceed to define the partial order that we will 
use throughout the paper. 
Let $A$ be an arc-diagram. We will identify
the vertices of $A$ with their labels. 
An {\em arc} in $A$ is a covering relation
in any of the labeled chains in $A$. 
If the arc denoted by $\alpha$ 
is a covering relation between the
vertices $i$ and $j$, then we write $\alpha = \{i,j\}$. 
In practice (while drawing the diagrams) we will 
always think of an arc as the graph of a 
connected concave down path in $\R^2$.
From this point of view, 
one of our most crucial conventions is that the arcs
of $A$ do not intersect each other if
they do not have to. We illustrate 
what we mean here in Figure~\ref{F:conventions}. 
If there is no possibility of continuously 
deforming two arcs $\alpha_1$ and $\alpha_2$ 
so that they do not intersect in $\R^2$, then 
they are said to {\em cross} each other. 
Otherwise, we call them {\em non-crossing} arcs.

\begin{figure}[h]
\begin{center}
\begin{tikzpicture}[scale=.5]

\begin{scope}[xshift=-7.5cm]
\node at (-4,0) {$\bullet$};
\node at (-2,0) {$\bullet$};
\node at (0,0) {$\bullet$};
\node at (2,0) {$\bullet$};
\node at (4,0) {$\bullet$};

\node at (-4,-0.5) {$1$};
\node at (-2,-0.5) {$2$};
\node at (0,-0.5) {$3$};
\node at (2,-0.5) {$4$};
\node at (4,-0.5) {$5$};
\node at (0,-2) {This is an arc-diagram.};
\draw[thick,blue,-] (-4,0) ..controls  (0,2) .. (4,0);
\draw[thick,blue,-] (-2,0) ..controls  (-1,1) .. (0,0);
\end{scope}

\begin{scope}[xshift=7.5cm]
\node at (-4,0) {$\bullet$};
\node at (-2,0) {$\bullet$};
\node at (0,0) {$\bullet$};
\node at (2,0) {$\bullet$};
\node at (4,0) {$\bullet$};

\node at (-4,-0.5) {$1$};
\node at (-2,-0.5) {$2$};
\node at (0,-0.5) {$3$};
\node at (2,-0.5) {$4$};
\node at (4,-0.5) {$5$};
\node at (0,-2) {This is not an arc-diagram.};
\draw[thick,blue,-] (-4,0) ..controls  (0,2) .. (4,0);
\draw[thick,blue,-] (-2,0) ..controls  (-1,3) .. (0,0);
\end{scope}
\end{tikzpicture}
\caption{Conventions.}
\label{F:conventions}
\end{center}
\end{figure}
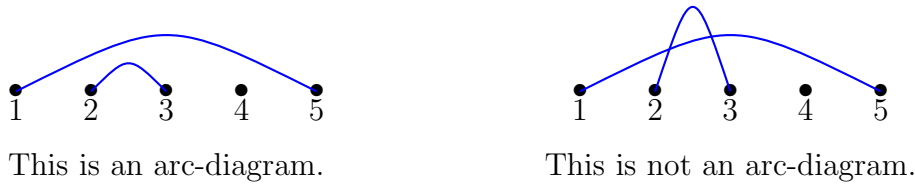

Before we proceed to explain our ordering on 
the arc-diagrams we will introduce a very useful 
function which will eventually lead us 
to a grading on our poset. 
This function is defined on all of the 
set of vertices, arcs, and chains of the arc-diagram. 
We will occasionally call a pair of non-crossing arcs
nested if both of the starting and the ending vertices
of one of the arcs stay below the other arc. 
 
\begin{Definition}\label{depth}
Let $A$ be an arc-diagram and let $\alpha$ be 
a vertex, or an arc, or a chain from $A$. 
The depth of $\alpha$, denoted by $depth(\alpha)$ 
is the total number of arcs ``above'' $\alpha$. 
\end{Definition}
Let us be more specific about what we mean by 
the word ``above'' in Definition~\ref{depth}: If 
$\alpha$ is a chain where $i$ is its leftmost 
vertex and $j$ is its rightmost vertex, then an arc $\{r,s\}$ 
is said to be above $\alpha$ if $r<i$ and $s>j$. 
For an example, see Figure~\ref{F:firstexample},
where every arc is of depth 0 and the vertex $4$
has depth 3.
\begin{figure}[h]
\begin{center}
\begin{tikzpicture}[scale=.4]
\node at (-14,0) {$\bullet$};
\node at (-12,0) {$\bullet$};
\node at (-10,0) {$\bullet$};
\node at (-8,0) {$\bullet$};
\node at (-6,0) {$\bullet$};
\node at (-4,0) {$\bullet$};
\node at (-2,0) {$\bullet$};
\node at (-14,-0.75) {$1$};
\node at (-12,-0.75) {$2$};
\node at (-10,-0.75) {$3$};
\node at (-8,-0.75) {$4$};
\node at (-6,-0.75) {$5$};
\node at (-4,-0.75) {$6$};
\node at (-2,-0.75) {$7$};
\draw[thick,blue,-] (-14,0) ..controls  (-10,2) .. (-6,0);
\draw[thick,blue,-] (-12,0) ..controls  (-9,2) .. (-4,0);
\draw[thick,blue,-] (-10,0) ..controls  (-6,2) .. (-2,0);
\end{tikzpicture}
\caption{$depth(\{2,6\})=0$.}
\label{F:firstexample}
\end{center}
\end{figure}
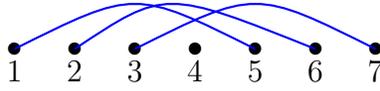
Obviously, for every arc-diagram the depths of 
the first and the last vertices are zero, that is 
$$
depth(1)=depth(n)=0.
$$ 
Another simple observation that will be useful 
in the sequel is that if an arc-diagram $A$ on 
$n$ vertices has $k$ arcs, 
then $A$ has exactly $n-k$ chains.
In this regard, let us point out that the number of set 
partitions in $\Pi_n$ with $k$ blocks,
hence the number of arc-diagrams in $\mc{A}_n$ 
with $k$ chains, is given by the Stirling numbers of the second kind;
it is easy to calculate them by using the simple recurrence
$$
S(n,k) = S(n-1,k-1) +  kS(n-1,k).
$$

Let $A$ and $B$ be two arc-diagrams on $n$ vertices. 
$B$ is said to cover $A$, and denoted by $A\prec B$, 
if it is obtained from $A$ by one of the following 
three operations:

\begin{enumerate}
\item[Rule 1.] {\em The shortening of an arc of $A$.}

In this operation we move exactly one endpoint
of an arc to another vertex so that the 
resulting arc is shortened as minimally as possible 
but the number of crossings does not change.
For example, see Figure~\ref{F:shortening}, where we
depict two examples. In the bottom example, 
the left endpoint of the arc $\{1,4\}$ is moved to the 
nearest available position, which is the vertex 
$3$. Indeed, 
there is already an arc which emanates to the right 
from the vertex $2$.
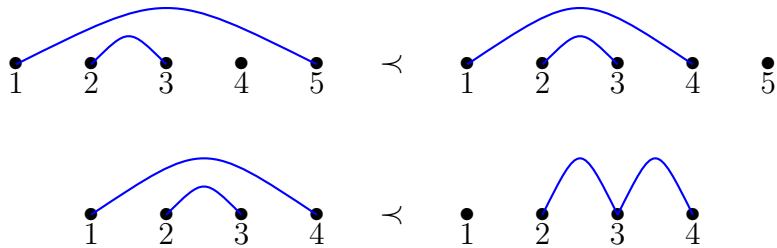
\begin{figure}[h]
\begin{center}
\begin{tikzpicture}[scale=.5]

\begin{scope}[yshift=2cm]
\node at (-10,0) {$\bullet$};
\node at (-8,0) {$\bullet$};
\node at (-6,0) {$\bullet$};
\node at (-4,0) {$\bullet$};
\node at (-2,0) {$\bullet$};
\node at (-10,-.5) {$1$};
\node at (-8,-.5) {$2$};
\node at (-6,-.5) {$3$};
\node at (-4,-.5) {$4$};
\node at (-2,-.5) {$5$};
\draw[thick,blue,-] (-10,0) ..controls  (-6,2) .. (-2,0);
\draw[thick,blue,-] (-8,0) ..controls  (-7,1) .. (-6,0);
\node at (0,0) {$\prec$};
\node at (2,0) {$\bullet$};
\node at (4,0) {$\bullet$};
\node at (6,0) {$\bullet$};
\node at (8,0) {$\bullet$};
\node at (10,0) {$\bullet$};
\node at (2,-.5) {$1$};
\node at (4,-.5) {$2$};
\node at (6,-.5) {$3$};
\node at (8,-.5) {$4$};
\node at (10,-.5) {$5$};
\draw[thick,blue,-] (2,0) ..controls  (5,2) .. (8,0);
\draw[thick,blue,-] (4,0) ..controls  (5,1) .. (6,0);
\end{scope}

\begin{scope}[yshift=-2cm]
\node at (-8,0) {$\bullet$};
\node at (-6,0) {$\bullet$};
\node at (-4,0) {$\bullet$};
\node at (-2,0) {$\bullet$};
\node at (-8,-.5) {$1$};
\node at (-6,-.5) {$2$};
\node at (-4,-.5) {$3$};
\node at (-2,-.5) {$4$};
\draw[thick,blue,-] (-8,0) ..controls  (-5,2) .. (-2,0);
\draw[thick,blue,-] (-6,0) ..controls  (-5,1) .. (-4,0);
\node at (0,0) {$\prec$};
\node at (2,0) {$\bullet$};
\node at (4,0) {$\bullet$};
\node at (6,0) {$\bullet$};
\node at (8,0) {$\bullet$};
\node at (2,-.5) {$1$};
\node at (4,-.5) {$2$};
\node at (6,-.5) {$3$};
\node at (8,-.5) {$4$};
\draw[thick,blue,-] (6,0) ..controls  (7,2) .. (8,0);
\draw[thick,blue,-] (4,0) ..controls  (5,2) .. (6,0);
\end{scope}
\end{tikzpicture}
\caption{Two examples for shortening.}
\label{F:shortening}
\end{center}
\end{figure}

\item[Rule 2.] {\em Deleting a crossing.}

In this operation we interchange
the rightmost endpoints of two crossing arcs
so that they become a pair of 
non-crossing and nested arcs;
we require in this operation that 
only one arc is deleted as a result of 
this operation. 
For example, in Figure~\ref{F:interchanging},
the endpoints of $\{1,5\}$ and $\{2,6\}$ 
are interchanged. 
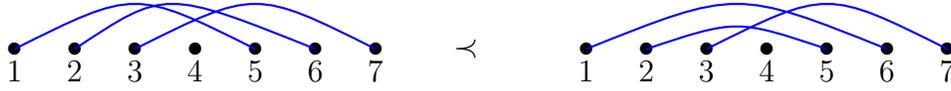
\begin{figure}[h]
\begin{center}
\begin{tikzpicture}[scale=.4]

\node at (-14,0) {$\bullet$};
\node at (-12,0) {$\bullet$};
\node at (-10,0) {$\bullet$};
\node at (-8,0) {$\bullet$};
\node at (-6,0) {$\bullet$};
\node at (-4,0) {$\bullet$};
\node at (-2,0) {$\bullet$};
\node at (-14,-0.75) {$1$};
\node at (-12,-0.75) {$2$};
\node at (-10,-0.75) {$3$};
\node at (-8,-0.75) {$4$};
\node at (-6,-0.75) {$5$};
\node at (-4,-0.75) {$6$};
\node at (-2,-0.75) {$7$};
\draw[thick,blue,-] (-14,0) ..controls  (-10,2) .. (-6,0);
\draw[thick,blue,-] (-12,0) ..controls  (-9,2) .. (-4,0);
\draw[thick,blue,-] (-10,0) ..controls  (-6,2) .. (-2,0);
\node at (1,0) {$\prec$};
\node at (5,0) {$\bullet$};
\node at (7,0) {$\bullet$};
\node at (9,0) {$\bullet$};
\node at (11,0) {$\bullet$};
\node at (13,0) {$\bullet$};
\node at (15,0) {$\bullet$};
\node at (17,0) {$\bullet$};
\node at (5,-0.75) {$1$};
\node at (7,-0.75) {$2$};
\node at (9,-0.75) {$3$};
\node at (11,-0.75) {$4$};
\node at (13,-0.75) {$5$};
\node at (15,-0.75) {$6$};
\node at (17,-0.75) {$7$};
\draw[thick,blue,-] (5,0) ..controls  (10,2) .. (15,0);
\draw[thick,blue,-] (7,0) ..controls  (10,1) .. (13,0);
\draw[thick,blue,-] (9,0) ..controls  (13,2) .. (17,0);
\end{tikzpicture}
\caption{Interchanging two endpoints.}
\label{F:interchanging}
\end{center}
\end{figure}

As a non-example, we consider 
$A=\{1,4\}\{2,5\}\{3,6\}$, which has three crossings. 
The removal of the crossing between $\{1,4\}$ and $\{3,6\}$
according to the rule that we described in the previous paragraph 
gives $A'=\{1,6\}\{2,5\}\{3,4\}$, which has no crossings.

\item[Rule 3.] {\em Adding a new arc.}

In this operation a new arc is introduced between two
vertices in such a way that the new arc is not under any other (older) arcs
and the endpoints of the new arc are as far from each other as possible.  
In Figure~\ref{F:adding} we depict two examples. In 
the former one the new arc is $\{1,6\}$ and in the latter the
new arc is $\{3,6\}$. 
\begin{figure}[h]
\begin{center}
\begin{tikzpicture}[scale=.5]
\begin{scope}[yshift=2cm]
\node at (-12,0) {$\bullet$};
\node at (-10,0) {$\bullet$};
\node at (-8,0) {$\bullet$};
\node at (-6,0) {$\bullet$};
\node at (-4,0) {$\bullet$};
\node at (-2,0) {$\bullet$};
\node at (-12,-0.5) {$1$};
\node at (-10,-0.5) {$2$};
\node at (-8,-0.5) {$3$};
\node at (-6,-0.5) {$4$};
\node at (-4,-0.5) {$5$};
\node at (-2,-0.5) {$6$};
\draw[thick,blue,-] (-10,0) ..controls  (-9,1) .. (-8,0);
\draw[thick,blue,-] (-8,0) ..controls  (-7,1) .. (-6,0);
\node at (0,0) {$\prec$};
\node at (2,0) {$\bullet$};
\node at (4,0) {$\bullet$};
\node at (6,0) {$\bullet$};
\node at (8,0) {$\bullet$};
\node at (10,0) {$\bullet$};
\node at (12,0) {$\bullet$};
\node at (2,-0.5) {$1$};
\node at (4,-0.5) {$2$};
\node at (6,-0.5) {$3$};
\node at (8,-0.5) {$4$};
\node at (10,-0.5) {$5$};
\node at (12,-0.5) {$6$};
\draw[thick,blue,-] (2,0) ..controls  (7,2) .. (12,0);
\draw[thick,blue,-] (4,0) ..controls  (5,1) .. (6,0);
\draw[thick,blue,-] (6,0) ..controls  (7,1) .. (8,0);
\end{scope}

\begin{scope}[yshift=-2cm]
\node at (-12,0) {$\bullet$};
\node at (-10,0) {$\bullet$};
\node at (-8,0) {$\bullet$};
\node at (-6,0) {$\bullet$};
\node at (-4,0) {$\bullet$};
\node at (-2,0) {$\bullet$};
\node at (-12,-0.5) {$1$};
\node at (-10,-0.5) {$2$};
\node at (-8,-0.5) {$3$};
\node at (-6,-0.5) {$4$};
\node at (-4,-0.5) {$5$};
\node at (-2,-0.5) {$6$};
\draw[thick,blue,-] (-12,0) ..controls  (-9,2) .. (-6,0);
\draw[thick,blue,-] (-10,0) ..controls  (-9,1) .. (-8,0);
\node at (0,0) {$\prec$};
\node at (2,0) {$\bullet$};
\node at (4,0) {$\bullet$};
\node at (6,0) {$\bullet$};
\node at (8,0) {$\bullet$};
\node at (10,0) {$\bullet$};
\node at (12,0) {$\bullet$};
\node at (2,-0.5) {$1$};
\node at (4,-0.5) {$2$};
\node at (6,-0.5) {$3$};
\node at (8,-0.5) {$4$};
\node at (10,-0.5) {$5$};
\node at (12,-0.5) {$6$};
\draw[thick,blue,-] (2,0) ..controls  (5,2) .. (8,0);
\draw[thick,blue,-] (4,0) ..controls  (5,1) .. (6,0);
\draw[thick,blue,-] (6,0) ..controls  (9,2) .. (12,0);
\end{scope}
\end{tikzpicture}
\caption{Two examples of adding a new arc.}
\label{F:adding}
\end{center}
\end{figure}
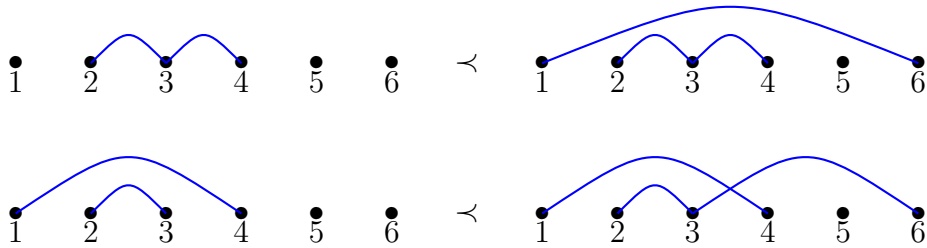
\end{enumerate}

From now on we will call the set $\mc{A}_n$ together 
with the transitive closure of the covering relations 
we just defined the arc-diagram poset and denote it by $(\mc{A}_n,\prec)$.

Next, we define our first combinatorial
statistic.
\begin{Definition}\label{D:statdd}
Let $A$ be an arc-diagram on $n$ vertices 
$v_1,\dots, v_n$ and with $k$ arcs 
$\alpha_1$, $\alpha_2$,...,$\alpha_k$. 
We define the depth-index of $A$, denoted by $\mathtt{t}(A)$ by the formula 
$$
\mathtt{t}(A)=\sum_{i=1}^k (n-i)
-\sum_{j=1}^n depth(v_j)+\sum_{m=1}^k depth(\alpha_m).
$$
\end{Definition}

One of the main results of our paper is the following statement.

\begin{Theorem}\label{T:main1}
For every positive integer $n$,
the arc-diagrams poset $(\mc{A}_n,\prec)$  
is a bounded, graded, and an EL-shellable 
poset. The depth-index function is  
the grading of $\mc{A}_n$.  
\end{Theorem}

The proof of our theorem 
is at least as interesting as its statement. 
To explain it, we venture outside of combinatorics.
Here we assume some familiarity with elementary algebraic geometry. 
Let $\mt{Mat}_n$ denote the linear algebraic monoid of $n\times n$ 
matrices defined over $\C$. The group of invertible elements,
also called the {\em unit group}, of $\mt{Mat}_n$ is the 
general linear group of invertible $n\times n$ matrices. 
The (standard) Borel subgroup of $\mt{GL}_n$, denoted by $\mt{B}_n$,
is the subgroup $\mt{B}_n\subset \mt{GL}_n$ 
consisting of upper triangular matrices only. 
Then the doubled Borel group $\mt{B}_n\times \mt{B}_n$ acts on 
matrices via 
\begin{align}\label{A:BxB action}
(b_1,b_2) \cdot x  = b_1 x b_2^{-1}\qquad (b_1,b_2\in \mt{B}_n,\ x \in \mt{Mat}_n)
\end{align}
Clearly, $\mt{GL}_n$ is stable under this action. 
By the special case of an important result of Renner~\cite{Renner86},  
it is known that the action (\ref{A:BxB action}) has finitely many orbits
and moreover the orbits of the action are parametrized 
by a finite inverse semigroup:
\begin{align}\label{A:BruhatRenner}
\mt{Mat}_n &= \bigsqcup_{\sigma \in R_n} \mt{B}_n \sigma \mt{B}_n,
\end{align}
where $R_n$ is the finite monoid consisting of $n\times n$ 0/1 
matrices with at most one 1 in each row and each column.
The monoid $R_n$ is called the rook monoid; its elements are called rooks.
(The nomenclature comes from the fact that 
the elements of $R_n$ are in bijection with the 
non-attacking rook placements on an $n\times n$ chessboard.)
The Bruhat-Chevalley-Renner ordering on $R_n$
is the partial ordering that is defined by 
\begin{align}\label{A:BCR}
\sigma \leq \tau \iff  \mt{B}_n \sigma \mt{B}_n \subseteq  \overline{\mt{B}_n \tau \mt{B}_n}
\end{align}
for $\sigma,\tau \in R_n$. This poset structure on $R_n$ is well studied,~\cite{CanRenner}. 
It is known that $(R_n,\leq)$ is a graded, bounded, EL-shellable poset, see~\cite{Can08}.

Towards a proof of Theorem~\ref{T:main1},  
we make use of an important algebraic submonoid of $\mt{Mat}_n$; 
it is the closure in Zariski topology of the Borel subgroup $\mt{B}_n$ in
$\mt{Mat}_n$. We will call $\BM{n}$ the (standard) Borel submonoid. 
The first systematic study of the theory of Borel submonoids
as a part of more general but interrelated theory of parabolic monoids is undertaken 
by Putcha in~\cite{Putcha06}. Here we are focusing on one extreme 
case only. 

The Borel submonoid $\BM{n}$ consists of all upper triangular $n\times n$ 
matrices with complex entries. To see this, we use the standard 
(semidirect product) decomposition
$$
\mt{B}_n = \mt{T}_n \mt{U}_n,
$$
where $\mt{T}_n$ is the maximal torus consisting of invertible diagonal 
matrices and $\mt{U}_n$ is the unipotent subgroup consisting of 
upper triangular unipotent matrices. It is easy to check that 
$\mt{U}_n$ is already closed in $\mt{Mat}_n$, therefore, 
the Borel submonoid is determined (generated) by its
submonoids $\overline{\mt{T}}_n$ and $\mt{U}_n$. 
Here, $\overline{\mt{T}}_n$ is the diagonal submonoid consisting 
of all diagonal matrices. Note that $\overline{\mt{T}}_n$ is an affine 
toric variety and there is a one-to-one correspondence 
between the cones of its defining ``fan'' and its set of idempotents.
(An idempotent in a monoid is an element $e$ such that 
$e^2 = id$.)

Let $M$ be a monoid and let $1_M$ 
denote its identity element. 
For us, a submonoid $N$ in a monoid $M$
is a subsemigroup $N\subset M$ such that $1_M\in N$.
In particular, $1_M$ is the identity element in $N$. 
Now, $\BM{n}$ is a submonoid of $\mt{Mat}_n$. 
Moreover, since it is closed under the two sided
action of $\mt{B}_n$, 
it has the induced Bruhat-Chevalley-Renner decomposition
\begin{align}\label{A:Rennerdecomposition}
\BM{n}=  \bigsqcup_{\sigma \in B_n} \mt{B}_n \sigma \mt{B}_n.
\end{align}
Here, $B_n$ is the set of all $n\times n$ rooks 
which are upper triangular in shape. 
Note that $B_n$ is a submonoid of $R_n$ according
to our definition. We call it the upper triangular rook monoid
(on $n$ letters). 
In Figure~\ref{F:B3} we depict 
the induced Bruhat-Chevalley-Renner ordering
on $B_3$. 

\begin{figure}[htp]
\begin{center}
\resizebox{14cm}{18cm}{
\begin{tikzpicture}[scale=.75]

\node at (0,30) (g1) {$\begin{bmatrix}   1 & 0 & 0 \\  0 & 1 & 0 \\  0 & 0 & 1 \\  \end{bmatrix}$};

\node at (-8,25) (f1) {$\begin{bmatrix}   1 & 0 & 0 \\  0 & 1 & 0 \\  0 & 0 & 0 \\  \end{bmatrix}$};
\node at (0,25) (f2) {$\begin{bmatrix}   1 & 0 & 0 \\  0 & 0 & 0 \\  0 & 0 & 1 \\   \end{bmatrix}$};
\node at (8,25) (f3) {$\begin{bmatrix}   0 & 0 & 0 \\  0 & 1 & 0 \\  0 & 0 & 1 \\  \end{bmatrix}$};

\node at (-8,20) (e1) {$\begin{bmatrix}  1 & 0 & 0 \\ 0 & 0 & 1\\  0 & 0 & 0 \\  \end{bmatrix}$};
\node at (0,20) (e2) {$\begin{bmatrix}   0 & 0 & 1 \\  0 & 1 & 0 \\  0 & 0 & 0 \\  \end{bmatrix}$};
\node at (8,20) (e3) {$\begin{bmatrix}   0 & 1 & 0 \\  0 & 0 & 0 \\  0 & 0 & 1 \\  \end{bmatrix}$};

\node at (-14,15) (d1) {$\begin{bmatrix}   1 & 0 & 0 \\  0 & 0 & 0 \\  0 & 0 & 0 \\  \end{bmatrix}$};
\node at (-5,15) (d2) {$\begin{bmatrix}  0 & 0 & 0 \\  0 & 1 & 0 \\ 0 & 0 & 0  \\  \end{bmatrix}$};
\node at (5,15) (d3) {$\begin{bmatrix}   0 & 1 & 0 \\  0 & 0 & 1 \\ 0 & 0 & 0 \\  \end{bmatrix}$};
\node at (14,15) (d4) {$\begin{bmatrix}   0 & 0 & 0 \\  0 & 0 & 0 \\  0 & 0 & 1 \\  \end{bmatrix}$};

\node at (-5,10) (c1) {$\begin{bmatrix}   0 & 1 & 0 \\  0 & 0 & 0 \\ 0 & 0 & 0 \\  \end{bmatrix}$};
\node at (5,10) (c2) {$\begin{bmatrix}   0 & 0 & 0 \\  0 & 0 & 1 \\ 0 & 0 & 0  \\  \end{bmatrix}$};

\node at (0,5) (b1) {$\begin{bmatrix}   0 & 0 & 1 \\  0 & 0 & 0 \\ 0 & 0 & 0  \\  \end{bmatrix}$};

\node at (0,0) (a) {$\begin{bmatrix}  0 & 0 & 0  \\ 0 & 0 & 0  \\ 0 & 0 & 0 \\  \end{bmatrix}$};

\draw[-, very thick] (a) to (b1);
\draw[-, very thick] (b1) to (c1);
\draw[-, very thick] (b1) to (c2);
\draw[-, very thick] (c1) to (d1);
\draw[-, very thick] (c1) to (d2);
\draw[-, very thick] (c1) to (d3);

\draw[-, very thick] (c2) to (d2);
\draw[-, very thick] (c2) to (d3);
\draw[-, very thick] (c2) to (d4);

\draw[-, very thick] (d1) to (e1);
\draw[-, very thick] (d2) to (e2);
\draw[-, very thick] (d4) to (e3);
\draw[-, very thick] (d3) to (e1);
\draw[-, very thick] (d3) to (e2);
\draw[-, very thick] (d3) to (e3);

\draw[-, very thick] (e1) to (f1);
\draw[-, very thick] (e1) to (f2);
\draw[-, very thick] (e2) to (f1);
\draw[-, very thick] (e2) to (f3);
\draw[-, very thick] (e3) to (f2);
\draw[-, very thick] (e3) to (f3);

\draw[-, very thick] (f1) to (g1);
\draw[-, very thick] (f2) to (g1);
\draw[-, very thick] (f3) to (g1);

\end{tikzpicture}
}
\end{center}
\caption{Bruhat-Chevalley-Renner order on $B_3$.}
\label{F:B3}
\end{figure}

Another subsemigroup that is very useful for our purposes
is the semigroup of all nilpotent rooks from $B_n$,
which we call the standard nilpotent rook monoid and denote by $B_n^{nil}$. 
We should point out that the identity element of $B_n^{nil}$ is not the same
as that of $B_n$. Nevertheless, $B_n^{nil}$ is a monoid.
In fact, for $n>0$, 
it is not difficult to see that $B_n^{nil}$ is isomorphic, 
as a monoid, to the upper triangular rook monoid $B_{n-1}$. 
By going through the same vein we observe that the 
semigroup of nilpotent elements in $\BM{n}$ is isomorphic as a monoid to $\BM{n-1}$. 
Moreover, this is an isomorphism of algebraic monoids.

The sets of idempotents of the monoids $\BM{n}$ 
and $\overline{\mt{T}}_n$ are the same and it consists
of $n\times n$ diagoanal matrices with 0/1 entries. 
Let us denote this common set of idempotents by $E_n$. 
It is not difficult to see that $E_n$ is a Boolean lattice 
with respect to the ordering
$$
e\leq f \iff e f = fe =e\ \hspace{.5cm} (e,f\in E_n),
$$
In particular, $E_n$ has $2^{n}$ elements. 
We denote by $E_{n,k}$ the set of idempotents from $E_n$ 
whose matrix rank is $k$ and we define the following subvariety 
the Borel monoid:
\begin{align}\label{D:Stirlingmonoid}
\mt{B}_{n,k} :=  \bigcup_{e\in E_{n,k}} \overline{\mt{B}_n e \mt{B}_n}.
\end{align}
Notice that except when $k\in \{0,n\}$, 
$\mt{B}_{n,k}$ is not irreducible as an algebraic variety. 
Obviously, $\mt{B}_{n,n}$ is equal to $\BM{n}$
and $\mt{B}_{n,0} = \overline{\mt{B}_n \cdot \mathbf{0} \cdot \mt{B}_n} =\{ \mathbf{0}\}$.

The proofs of the following observations will be given in the sequel.
\begin{enumerate}
\item for $k=0,\dots, n$, the number of irreducible components of $\mt{B}_{n,k}$
is ${n \choose k}$ and they are all equal dimensional.
\item $\mt{B}_{n,k}$'s form a flag 
$\{ \mathbf{0}\} = \mt{B}_{n,0} \subset \mt{B}_{n,1} \subset \cdots \subset \mt{B}_{n,n-1} \subset 
\mt{B}_{n,n} = \BM{n}$.
\item each $\mt{B}_{n,k}$ ($k=0,\dots, n$) has the structure of an algebraic 
semigroup.
\item each $\mt{B}_{n,k}$ ($k=0,\dots, n$)
has a Renner decomposition 
\begin{align}\label{A:RennerStirlingdecomposition}
\mt{B}_{n,k}=  \bigsqcup_{\sigma \in B_{n,k}} \mt{B}_n \sigma \mt{B}_n,
\end{align}
where $B_{n,k}$ is a finite subsemigroup of $B_n$ and it consists of 
rooks whose matrix rank is at most $k$. 
Moreover, with respect to induced Bruhat-Chevalley-Renner ordering
the poset $(B_{n,k},\leq )$ is a 
union of lower intervals of equal lengths in $B_n$.
\item The subsemigroups $B_{n,k}\subset B_n$ form a 
flag $\{\mathbf{0}\}\subset B_{n,1}\subset \cdots \subset B_{n,n}=B_n$
and moreover the number of elements of $B_{n,k}-B_{n,k-1}$ is 
given by the Stirling number $S(n+1,n+1-k)$. 
\item The Bruhat-Chevalley-Renner ordering
restricted to the subsets of the form $B_{n,k}-B_{n,k-1}$
(for $k=1,\dots, n$) is graded
with a minimum and there are ${n \choose k}$ maximal elements. 
Each maximal interval in this poset is an interval in $B_n$,
therefore, it is an EL-shellable poset. 
\end{enumerate}

As an application of our study of the 
Bruhat-Chevalley-Renner ordering on $B_{n,k}$'s 
we will prove the following theorem, which, in turn, 
will give us the proof of Theorem~\ref{T:main1}. 
Indeed, the poset $(B_n^{nil},\leq)$ is a lower interval
in the rook monoid, and $R_n$ is known
to be an EL-shellable poset. 

\begin{Theorem}\label{T:Borelmonoid}
The arc-diagram poset $(\mc{A}_n,\prec)$ is isomorphic to $(B_n^{nil},\leq)$.
\end{Theorem}

Next, we show that the arc-diagram poset is a 
disjoint union of EL-shellable subposets, 
which are not necessarily intervals. The cardinalities of 
these subposets will be given by the Stirling numbers of 
the second kind.

\begin{Theorem}\label{T:Borelmonoid2}
If $\mc{A}_{n,k}$ denotes the set of arc-diagrams 
with $n-k$ chains, then $(\mc{A}_{n,k},\prec)$
is a graded EL-shellable poset with a unique minimum
and ${n \choose k}$ maximum elements. 
\end{Theorem}

\begin{Definition}
The $(n,k)$-th Stirling poset is the poset $(\mc{A}_{n,k},\prec)$. 
By abusing notation, we will denote it by $\mc{A}_{n,k}$. 
\end{Definition}

To contrast $\mc{A}_{n,k}$ with the corresponding subposet 
in the refinement ordering on set partitions, 
let us mention that any two unequal set partitions of $\{1,\dots, n\}$
with the same number of blocks are not comparable. 
In other words, the collection of 
arc-diagrams with the same number of chains do not form an interesting 
poset with respect to refinement ordering.  
On the other hand, similarly to the refinement ordering,
in $(\mc{A}_n,\prec)$, the Stirling subposets have a hierarchy in the 
sense that $\mc{A}_{n,k}$ lies above $\mc{A}_{n,k-1}$. Indeed,  
if $x$ and $y$ are two maximal elements from $\mc{A}_{n,k}$
and $\mc{A}_{n,k-1}$, respectively, then $\mathtt{t}(x) - \mathtt{t}(y) = n-k$.
From a similar vein, if $x_0$ and $y_0$ denotes, respectively, 
the minimum elements of $\mc{A}_{n,k}$ and $\mc{A}_{n,k-1}$, 
then $\mathtt{t}(x_0)-\mathtt{t}(y_0) = k$. 
\vspace{.5cm}

It is not difficult to 
see that when $k=1$, $\mc{A}_{n,1}$ is 
the ``fish net'' as in Figure~\ref{F:fishnet},
hence every interval in $\mc{A}_{n,1}$ is a lattice. 
As $k$ increases, $\mc{A}_{n,k}$ becomes 
more complicated. 
Nevertheless, it is a pleasantly 
surprising fact that $\mc{A}_{n,2}$ 
is a lattice as well. 
The smallest integer $n$ for which $\mc{A}_{n,k}$
has a non-lattice subinterval is $n=5$. See Figure~\ref{F:nonlattice}.
\begin{figure}[h]
\begin{center}
\begin{tikzpicture}[scale=.4]
\begin{scope}[xshift=-12cm, yshift=8cm]
\node at (-2,0) {$\bullet$}; 
\node at (-1,0) {$\bullet$}; 
\node at (0,0) {$\bullet$}; 
\node at (1,0) {$\bullet$}; 
\node at (2,0) {$\bullet$}; 
\draw[thick,blue,-] (1,0) ..controls  (1.5,.5) .. (2,0);
\end{scope}
\begin{scope}[xshift=-4cm, yshift=8cm]
\node at (-2,0) {$\bullet$}; 
\node at (-1,0) {$\bullet$}; 
\node at (0,0) {$\bullet$}; 
\node at (1,0) {$\bullet$}; 
\node at (2,0) {$\bullet$}; 
\draw[thick,blue,-] (0,0) ..controls  (.5,.5) .. (1,0);
\end{scope}
\begin{scope}[xshift=4cm, yshift=8cm]
\node at (-2,0) {$\bullet$}; 
\node at (-1,0) {$\bullet$}; 
\node at (0,0) {$\bullet$}; 
\node at (1,0) {$\bullet$}; 
\node at (2,0) {$\bullet$}; 
\draw[thick,blue,-] (-1,0) ..controls  (-.5,.5) .. (0,0);
\end{scope}
\begin{scope}[xshift=12cm, yshift=8cm]
\node at (-2,0) {$\bullet$}; 
\node at (-1,0) {$\bullet$}; 
\node at (0,0) {$\bullet$}; 
\node at (1,0) {$\bullet$}; 
\node at (2,0) {$\bullet$}; 
\draw[thick,blue,-] (-2,0) ..controls  (-1.5,.5) .. (-1,0);
\end{scope}

\begin{scope}[xshift=-8cm, yshift=4cm]
\node at (-2,0) {$\bullet$}; 
\node at (-1,0) {$\bullet$}; 
\node at (0,0) {$\bullet$}; 
\node at (1,0) {$\bullet$}; 
\node at (2,0) {$\bullet$}; 
\draw[thick,blue,-] (0,0) ..controls  (1,1) .. (2,0);
\end{scope}
\begin{scope}[xshift=0cm, yshift=4cm]
\node at (-2,0) {$\bullet$}; 
\node at (-1,0) {$\bullet$}; 
\node at (0,0) {$\bullet$}; 
\node at (1,0) {$\bullet$}; 
\node at (2,0) {$\bullet$}; 
\draw[thick,blue,-] (-1,0) ..controls  (0,1) .. (1,0);
\end{scope}
\begin{scope}[xshift=8cm, yshift=4cm]
\node at (-2,0) {$\bullet$}; 
\node at (-1,0) {$\bullet$}; 
\node at (0,0) {$\bullet$}; 
\node at (1,0) {$\bullet$}; 
\node at (2,0) {$\bullet$}; 
\draw[thick,blue,-] (-2,0) ..controls  (-1,1) .. (0,0);
\end{scope}

\begin{scope}[xshift=-4cm, yshift=0cm]
\node at (-2,0) {$\bullet$}; 
\node at (-1,0) {$\bullet$}; 
\node at (0,0) {$\bullet$}; 
\node at (1,0) {$\bullet$}; 
\node at (2,0) {$\bullet$}; 
\draw[thick,blue,-] (-1,0) ..controls  (.5,1.25) .. (2,0);
\end{scope}
\begin{scope}[xshift=4cm, yshift=0cm]
\node at (-2,0) {$\bullet$}; 
\node at (-1,0) {$\bullet$}; 
\node at (0,0) {$\bullet$}; 
\node at (1,0) {$\bullet$}; 
\node at (2,0) {$\bullet$}; 
\draw[thick,blue,-] (-2,0) ..controls  (-.5,1.25) .. (1,0);
\end{scope}

\begin{scope}[xshift=0cm, yshift=-4cm]
\node at (-2,0) {$\bullet$}; 
\node at (-1,0) {$\bullet$}; 
\node at (0,0) {$\bullet$}; 
\node at (1,0) {$\bullet$}; 
\node at (2,0) {$\bullet$}; 
\draw[thick,blue,-] (-2,0) ..controls  (0,1.25) .. (2,0);
\end{scope}

\draw[black,-] (-12, 8) to (0,-4);
\draw[black,-] (12, 8) to (0,-4);
\draw[black,-] (-4, 8) to (4,0);
\draw[black,-] (4, 8) to (-4,0);
\draw[black,-] (4, 8) to (8,4);
\draw[black,-] (-4, 8) to (-8,4);
\end{tikzpicture}
\end{center}
\caption{The Stirling poset $\mc{A}_{5,1}$.}
\label{F:fishnet}
\end{figure}

\begin{Theorem}\label{T:(n,2) is boolean}
For all integers $n \geq 2$, the $(n,2)$-th Stirling poset
$\mc{A}_{n,2}$ is isomorphic to $B(n-1) - \{\{1,\dots, n-1\}\}$,
where $B(n-1)$ is the boolean lattice 
of all subsets of $\{1,\dots, n-1\}$.
\end{Theorem}

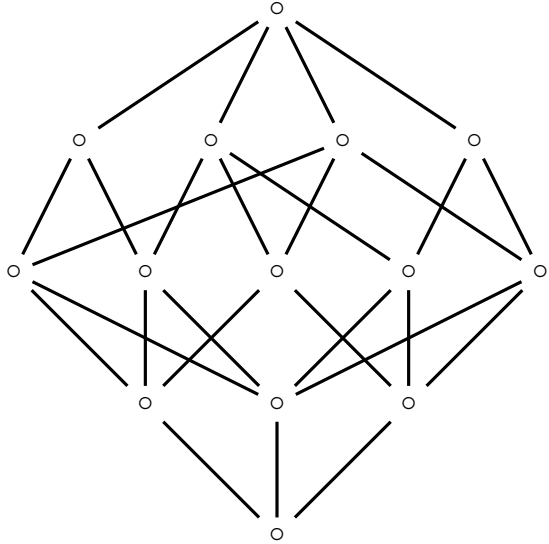
\begin{figure}[htp]
\begin{center}
\begin{tikzpicture}[scale=.35]
\node at (0,20) (e1) {$\circ$};
\node at (-7.5,15) (d1) {$\circ$};
\node at (-2.5,15) (d2) {$\circ$};
\node at (2.5,15) (d3) {$\circ$};
\node at (7.5,15) (d4) {$\circ$};
\node at (-10,10) (c1) {$\circ$};
\node at (-5,10) (c2) {$\circ$};
\node at (0,10) (c3) {$\circ$};
\node at (5,10) (c4) {$\circ$};
\node at (10,10) (c5) {$\circ$};
\node at (-5,5) (b1) {$\circ$};
\node at (0,5) (b2) {$\circ$};
\node at (5,5) (b3) {$\circ$};
\node at (0,0) (a) {$\circ$};

\draw[-, very thick] (a) to (b1);
\draw[-, very thick] (a) to (b2);
\draw[-, very thick] (a) to (b3);
\draw[-, very thick] (b1) to (c1);
\draw[-, very thick] (b1) to (c2);
\draw[-, very thick] (b1) to (c3);
\draw[-, very thick] (b2) to (c1);
\draw[-, very thick] (b2) to (c2);
\draw[-, very thick] (b2) to (c4);
\draw[-, very thick] (b2) to (c5);
\draw[-, very thick] (b3) to (c3);
\draw[-, very thick] (b3) to (c4);
\draw[-, very thick] (b3) to (c5);
\draw[-, very thick] (c1) to (d1);
\draw[-, very thick] (c1) to (d3);
\draw[-, very thick] (c2) to (d1);
\draw[-, very thick] (c2) to (d2);
\draw[-, very thick] (c3) to (d2);
\draw[-, very thick] (c3) to (d3);
\draw[-, very thick] (c4) to (d2);
\draw[-, very thick] (c4) to (d4);
\draw[-, very thick] (c5) to (d3);
\draw[-, very thick] (c5) to (d4);
\draw[-, very thick] (d1) to (e1);
\draw[-, very thick] (d2) to (e1);
\draw[-, very thick] (d3) to (e1);
\draw[-, very thick] (d4) to (e1);

\end{tikzpicture}
\end{center}
\caption{A non-lattice maximal subinterval in $\mc{A}_{5,3}$.}
\label{F:nonlattice}
\end{figure}

The arc-diagram poset $\mc{A}_n$ contains  
many interesting (Stirling) posets.
But it has more in it; we will justify our statement in our next result.
Let us state the relevant terminology here. 
By a partial flag variety we mean a quotient variety 
of the form $\mt{GL}_n/\mt{P}$, where $\mt{P}$ is a 
closed subgroup containing $\mt{B}_n$. A Schubert 
variety is the Zariski closure of an orbit of $\mt{B}_n$
on the partial flag variety. Note that $\mt{B}_n$ 
acts on $\mt{GL}_n/\mt{P}$ via left multiplication.

\begin{figure}[h]
\begin{center}
\resizebox{8cm}{6cm}{
\begin{tikzpicture}

\begin{scope}[scale=.45, yshift=6cm]
\node at (-10,0) {$W(n) = $};
\node at (-6,0) {$\bullet$};
\node at (-4,0) {$\bullet$};
\node at (-2,0) {$\dots$};
\node at (0,0) {$\bullet$};
\node at (2,0) {$\bullet$};
\node at (4,0) {$\bullet$};
\node at (6,0) {$\bullet$};
\node at (8,0) {$\dots$};
\node at (10,0) {$\bullet$};
\node at (-6,-.65) {$1$};
\node at (-4,-.65) {$2$};
\node at (-.45,-.65) {$n$};
\node at (1.5,-.65) {$n+1$};
\node at (4,-.65) {$n+2$};
\node at (6.5,-.65) {$n+3$};
\node at (10.5,-.65) {$2n$};
\draw[thick,blue,-] (-6,0) ..controls  (-5,1) .. (-4,0);
\draw[thick,blue,-] (0,0) ..controls  (1,1) .. (2,0);
\end{scope}

\begin{scope}[scale=.45, yshift=3cm]
\node at (-6,0) {$Z(n) = $};
\node at (0,0) {$\dots$};
\node at (-2,0) {$\bullet$};
\node at (-4,0) {$\bullet$};
\node at (2,0) {$\bullet$};
\node at (-4,-.65) {$1$};
\node at (-2,-.65) {$2$};
\node at (2,-.65) {$2n$};
\end{scope}

\begin{scope}[scale=.45, yshift=-2.5cm]
\node at (-9,0) {$X(n) = $};
\node at (-6,0) {$\bullet$};
\node at (-4,0) {$\bullet$};
\node at (-2,0) {$\bullet$};
\node at (0,0) {$\dots$};
\node at (2,0) {$\bullet$};
\node at (4,0) {$\bullet$};
\node at (6,0) {$\bullet$};
\node at (-6,-.65) {$1$};
\node at (-4,-.65) {$2$};
\node at (-2,-.65) {$3$};
\node at (1.5,-.65) {$2n-2$};
\node at (4.4,-.65) {$2n-1$};
\node at (6.5,-.65) {$2n$};
\draw[thick,blue,-] (-6,0) ..controls  (0,4) .. (6,0);
\draw[thick,blue,-] (-4,0) ..controls  (0,3) .. (4,0);
\draw[thick,blue,-] (-2,0) ..controls  (0,2) .. (2,0);
\end{scope}

\begin{scope}[scale=.45, xshift=8cm, yshift=-8cm]
\node at (-16,0) {$Y(n) = $};
\node at (-14,0) {$\bullet$};
\node at (-12,0) {$\bullet$};
\node at (-10,0) {$\dots$};
\node at (-8,0) {$\bullet$};
\node at (-6,0) {$\bullet$};
\node at (-4,0) {$\bullet$};
\node at (-2,0) {$\dots$};
\node at (0,0) {$\bullet$};
\node at (-14,-.65) {$1$};
\node at (-12,-.65) {$2$};
\node at (-8.5,-.65) {$n$};
\node at (-6,-.65) {$n+1$};
\node at (-3.25,-.65) {$n+2$};
\node at (0,-.65) {$2n$};
\draw[thick,blue,-] (-14,0) ..controls  (-10,3) .. (-6,0);
\draw[thick,blue,-] (-12,0) ..controls  (-8,3) .. (-4,0);
\draw[thick,blue,-] (-8,0) ..controls  (-4,3) .. (0,0);
\end{scope}

\end{tikzpicture}
}
\end{center}
\caption{}
\label{F:allofthem}
\end{figure}

\begin{Theorem}\label{T:intervals}
Let $X(n),Y(n)$, and $Z(n)$, and be as in Figure~\ref{F:allofthem}. 
In addition, let $S_n$ denote the symmetric group on $\{1,\dots,n\}$. 
Then the following statements hold true: 
\begin{enumerate}
\item The interval $([Y(n),X(n)],\prec)$ in $\mc{A}_{2n}$ is isomorphic to $(S_n,\leq)$. 
\item The interval $([Z(n),Y(n)],\prec)$ in $\mc{A}_{2n}$ and is isomorphic to $(B_n,\leq)$.
\item The interval $([Z(n),X(n)],\prec)$ in $\mc{A}_{2n}$ is isomorphic to $(R_n,\leq)$. 
\item The interval $([Y(n),W(n)],\prec)$ in $\mc{A}_{2n}$ is isomorphic to the inclusion 
poset of Borel orbit closures in a Schubert variety.
\end{enumerate}
\end{Theorem}

Our final remark concerns the length generating function 
of the $(n,k)$-th Stirling poset. Let us denote by $\mathtt{t}_k$
the length function on $\mc{A}_{n,k}$.
Clearly, $\mathtt{t}_k$ is equal to an appropriate shift of $\mathtt{t}$. More precisely, 
let $A$ be an element from $\mc{A}_{n,k}$.  
If we view $A$ as an element of $\mc{A}_n$, then 
it is clear that $\mathtt{t}(A)=\mathtt{t}_k(A) + {k\choose 2}$
since the unique minimum of $\mc{A}_{n,k}$ has depth-index ${k \choose 2}$.  
To be able to treat all length generating functions $\mathtt{t}_k$ ($k=0,\dots, n$) together, 
we define 
\begin{align}\label{A:analog of stirling}
{n \bb k} := \sum_{A \in \mc{A}_{n,k}} q^{\mathtt{t}(A)}.
\end{align}
Obviously, (\ref{A:analog of stirling}) is a $q$-analog of the Stirling numbes of the second kind.

\begin{Theorem}\label{T:gen func}
For positive integers $n$ and $k$ such that $0\leq k \leq n+1$
the following recurrence holds true:
\begin{align*}
{n+1 \bb k} = q^{k} {n \bb k} + [n+1-k]_q q^{k} {n\bb k-1},
\end{align*}
where $[k]_q$ is the polynomial $1+q+\cdots + q^{k-1}$. 
The initial conditions are ${m \bb 0 } = 1$ for all $m\in \N$.
In addition, we assume that 
$
{m \bb k } = 0\; \text{ if $k<0$ or $k > m$.}
$
\end{Theorem}

For various $(p,q)$-analogs of Stirling numbers of the second kind, 
see Wachs and White's influential article~\cite{WachsWhite}. 
Also, for many other poset theoretic properties of set-partitions 
(under refinement ordering) we recommend the excellent expository
article~\cite{WachsExpository} by Wachs.

We now describe the structure of our paper. 
We designed Section~\ref{S:Preliminaries} so that it gives 
the necessary background for the subsequent sections. In particular,
we review the concepts of EL-shellability, rook monoid, 
and recall some characterizations of the Bruhat-Chevalley-Renner
ordering together with its length functions. 
The Section~\ref{S:proofof2} is devoted to a proof of 
Theorem~\ref{T:Borelmonoid} and to a proof of the first 
part of Theorem~\ref{T:main1}. The second part of Theorem~\ref{T:main1}
is given in the subsequent Section~\ref{S:Stats},
where we prove that the length function on $B_n$ is equivalent
to the depth-index function $\mathtt{t}$. In the same section,
we introduce another statistic, denoted by $\mathtt{c}$, 
and called the ``crossing-index of an arc-diagram.''
We prove that $\mathtt{c} = \mathtt{t}$. 
Section~\ref{S:Subposets} is the most algebro-geometric section 
of our paper. We prove six properties that we mentioned above 
about the variety $\BM{n}$ and its subvarieties. 
The proof of Theorem~\ref{T:(n,2) is boolean} is recorded therein 
as well. In Section~\ref{S:Intervals} we prove 
Theorem~\ref{T:intervals} which is a characterization of some 
special subintervals of $\mc{A}_n$. 
Finally, in Section~\ref{S:Recurrence} we analyze the 
length generating function of the posets $\mc{A}_{n,k}$ 
and prove Theorem~\ref{T:gen func}.

\section{Preliminaries}\label{S:Preliminaries} 


\subsection{Set partitions}
\label{S:Preliminaries1}

Although we do not use this fact in the sequel, 
let us mention that 
the number of set partitions from $\Pi_n$ 
is given by the $n$-th Bell number,
which is denoted by $b_n$. 
The exponential generating series of $b_n$ is given by $e^{e^x-1}$. 

\subsection{EL-shellable posets}

A finite graded poset $P$ with a maximum and a minimum element is called {\em EL-shellable}, 
if there exists a map $f=f_{\varGamma}: C(P) \rightarrow \varGamma$ between the 
set of covering relations $C(P)$ of $P$ into a totally ordered set $\varGamma$ satisfying 
\begin{enumerate}
\item in every interval $[x,y] \subseteq P$ of length $k>0$ there exists a unique saturated chain 
$\mathcal{C}:\ x_0=x < x_1 < \cdots < x_{k-1} < x_k=y$ such that the entries of the sequence
\begin{align}\label{JordanHolder}
f(\mathcal{C}) = (f(x_0,x_1), f(x_1,x_2), \dots , f(x_{k-1},x_k))
\end{align}
is weakly increasing.
\item The sequence $f(\mathcal{C})$ of the unique chain 
$\mathcal{C}$ from (1) is the lexicographically 
smallest among all sequences of the form 
$(f(x_0,x_1'), f(x_1',x_2'), \dots , f(x_{k-1}',x_k))$, 
where $x_0 < x_1' < \cdots < x_{k-1}' < x_k$. 
\end{enumerate}

The {\em order complex} of a poset $P$ is the abstract simplicial complex $\Delta(P)$ 
whose simplicies are the chains in $P$. 
For an EL-shellable poset the order complex is shellable, 
in particular it implies that $\Delta(P)$ is Cohen-Macaulay \cite{Bjorner80}. 
These, of course, are among the most desirable properties of a topological space.

\begin{Remark}
In the sequel, specifically for the Stirling posets, we will relax the unique maximum 
element condition in the definition of EL-shellability. 
\end{Remark}

\begin{Remark}
There are various lexicographic shellability conditions 
in the literature and the EL-shellability defined here is among the stronger ones. 
See~\cite{BjornerWachs}
\end{Remark}

\subsection{Algebraic monoids}

In this section we provide the bare minimum background on reductive monoids 
to help the reader to understand the geometric/group theoretic 
angle of our work. We start with defining (more general) algebraic monoids. 

Let $k$ be an algebraically closed field and let $M$ be an irreducible variety
with a morphism $a:M\times M \rightarrow M$ and an element $e\in M$ 
such that 
\begin{itemize}
\item $a(x,a(y,z))= a(a(x,y),z)$ for all $x,y,z$ from $M$;
\item $a(e,x)=a(x,e)=x$ for all $x$ from $M$.
\end{itemize}
Thus, $M$ is an algebraic monoid. Let $G$ denote the group of invertible 
elements in $M$. If $G$ is a reductive algebraic group, then $M$ is called
a reductive monoid. 

Let $E(M)$ denote the set of idempotents of $M$. There is an important
partial order on $E(M)$ that is defined by 
\begin{align}\label{A:idempotentorder}
e\leq f \iff ef = e = fe. 
\end{align}

For reductive monoids, there exists a finite sublattice 
$\varLambda \subset E(M)$ 
such that 
\begin{enumerate}
\item $M = \bigsqcup_{e\in \varLambda} G e G$;
\item $e \leq f \iff G e G \subset \overline{GfG}$ for $e,f\in \varLambda$.
\end{enumerate}
In the second item, the bar stands for closure in Zariski topology.
The lattice $\varLambda$ (unique up to conjugation) is called the cross 
section lattice of $M$ and it uniquely determines many important subgroups
of $G$. For example, 
\begin{align}\label{A:Borel}
B = \{ g\in G:\ ge = ege \ \text{ for all } e\in \varLambda \}
\end{align}
is a Borel subgroup and its opposite $B^-$ is given by
\begin{align}\label{A:opBorel}
B^- = \{ g\in G:\ eg = ege \ \text{ for all } e\in \varLambda \}.
\end{align}
The maximal torus of $B$ is 
\begin{align}\label{A:maxtorus}
T = \{ g\in G:\ ge = eg \ \text{ for all } e\in \varLambda \}
\end{align}

Let $N_G(T)$ denote the normalizer of $T$ in $G$ and let $W$ denote $N_G(T)/T$,
the {\em Weyl group} of $(G,T)$. Let $S\subset W$ be a generating system consisting
of simple reflections. It is well known that $W$ is a graded poset with the rank function
$\ell : W  \longrightarrow \Z $ defined by 
\begin{align}\label{A:rank function}
\ell (w) = \text{ dimension of the image of $BwB$ in $G/B$ }.
\end{align}

The reductive monoid $M$ has the {\em Renner decomposition} 
$$
M = \bigsqcup_{\sigma \in R(M)} B \sigma B 
$$
where $R(M)$ is a finite inverse semigroup having $W$ as its unit group. 
In fact, $R(M)= \overline{N_G(T)}/T$. 
We will call $R$ the Renner monoid of $(M,T)$. 
Extending the Bruhat-Chevalley 
ordering on the $W$, there is a natural graded partial order on the Renner monoid:
\begin{align}\label{A:Renner order}
\sigma \leq \tau \iff B\sigma B \subset \overline{B\tau B}
\end{align}
for $\sigma,\tau \in R(M)$. We will call (\ref{A:Renner order}) the Bruhat-Renner-Chevalley ordering.
\begin{Remark}\label{R:Renner=idempotent order}
If $\sigma$ and $\tau$ are two idempotents from $R(M)$, then 
$\sigma \leq \tau$ in (\ref{A:Renner order}) if and only if $\sigma \leq \tau$ in (\ref{A:idempotentorder}).
\end{Remark}

The Renner monoid $R(M)$, is an inverse semigroup. 
This means that for each element $x$ of $R(M)$ there 
exists a corresponding $x^*\in R(M)$ such that $x x^* x = x$ and $x^* x x^* = x^*$. 
An important commonality between all such monoids is that they admit a faithful linear semigroup
representation. More precisely, let $R_n$ denote the Renner monoid $R(\mt{Mat}_n)$. 
It is well known that $R_n$ is isomorphic to the inverse semigroup of 
all injective partial transformations on the set $\{1,\dots, n\}$.
Furthermore, if $R$ is an inverse semigroup, then for some $n$ there exists an injective 
semigroup homomorphism $\phi: R\rightarrow R_n$. 
Following our terminology from the introduction, will call $R_n$ the rook monoid 
since its elements can be viewed as rook placements on an $n\times n$ grid, 
where the nonzero entries of an element of $R_n$ are viewed as the non-attacking rook placements.
It is also possible to represent the elements of $R_n$ in one-line notation
and describe the covering relations of the Bruhat-Renner-Chevalley ordering in this context. 
We will briefly review this development.

Recall from \cite{Renner86} that the rank function on $R_n$ is given by 
\begin{equation*}
\ell( x )= \dim (B x B),\  x \in R_n.
\end{equation*}
There is a combinatorial formula for $\ell(x)$, $x\in R_n$.
To explain we represent elements of $R_n$ by $n$-tuples. 
For $x=(x_{ij}) \in R_n$ we define the sequence $(a_1,\dots,a_n)$ by
\begin{equation}\label{E:oneline}
a_j = 
\begin{cases}
0  &\text{if the $j$-th column consists of zeros,}\\
i &\text{if $x_{ij}=1$.}
\end{cases}
\end{equation}
By abuse of notation, we denote both the matrix and the sequence $(a_1,\dots,a_n)$ by $x$.
For example, the associated sequence of the partial permutation matrix 
\begin{equation*}
x=\begin{pmatrix}  
0 & 0 & 0 & 0 \\
0 & 0 & 0 & 0 \\
1 & 0 & 0 & 0 \\ 
0 & 0 & 1 & 0
\end{pmatrix}
\end{equation*}
is $x=(3,0,4,0)$.

Let $x=(a_1,\dots.,a_n)\in R_n$.  A pair $(i,j)$ of indices $1\leq i<j \leq n$ is called a {\em coinversion pair} for $x$, 
if $0< a_i < a_j$. We denote the number of conversion pairs of $x$ by $coinv(x)$. 
\begin{Example}
Let $x=(4,0,2,3)$. Then, the only coinversion pair for $x$ is $(3,4)$. Therefore, $coinv(x)=1$.
\end{Example}

In \cite{CanRenner}, it is shown that the dimension $\ell(x)=\dim(BxB)$ of an orbit $B x B$, $x\in R_n$ 
is given by
\begin{equation}\label{E:dimforminv}
\ell(x)  = (\sum_{i=1}^n a_i^*)  - coinv(x),\ \text{where}\
a_i^* =
\begin{cases}
a_i+n-i,  & \text{if}\  a_i\neq 0 \\
0,   & \text{if}\ a_i=0
\end{cases}
\end{equation}
Reformulating (\ref{E:dimforminv}) gives 
\begin{Proposition}\label{P: sum inv}
Let $x=(a_1,\dots, a_n)\in R_n$. Then 
\begin{align*}
\ell( x) = \sum a_i + inv (x),
\end{align*}
where $inv(x) = |\{(i,j):\ 1\leq i< j \leq n,\ a_i>a_j \} |$.
\end{Proposition}
 
As a corollary of Proposition \ref{P: sum inv} we have 
\begin{Corollary}
Let $w=(a_1,\dots,a_n) \in S_n$ be a permutation. Then
$\ell(w) = {n+1 \choose 2}+inv(w)$.
\end{Corollary}

First concrete description of the Bruhat-Chevalley-Renner ordering on $R_n$ is given in \cite{PPR97}: 
\begin{Theorem}\label{T:PPR} 
Let $x = (a_1,\dots,a_n)$, $y=(b_1,\dots,b_n) \in R_n$. The Bruhat-Renner-Chevalley ordering on 
$R_n$ is the smallest partial order on $R_n$ generated by declaring 
$x \leq y$ if either 
\begin{enumerate}
\item there exists an $1 \leq i \leq n$ such that $b_i> a_i$ and $b_j = a_j$ for all $j\neq i$, or
\item  there exist  $1 \leq i < j \leq n$ such that $b_i=a_j,\ b_j=a_i$ with $b_i > b_j$, 
and for all $k\notin \{i,j\}$, $b_k = a_k$.
\end{enumerate}
\end{Theorem}

The covering relations of the order are analyzed in detail in \cite{CanRenner}, 
and the following two lemmas are found out to be very useful.

\begin{Lemma} \label{L:PPRcovering0}
Let $x=(a_1,\dots,a_n)$ and $y = (b_1,\dots,b_n)$ be elements of $R_n$. 
Suppose that $a_k=b_k$ for all $k =\{ 1,\dots,\widehat{i},\dots,n\}$ and $a_i < b_i$. 
Then, $\ell(y) = \ell( x)+ 1$ if and only if either
\begin{enumerate}
\item $0=a_i$, $b_i=1$ and $a_j=b_j > 0$ for all $j > i$, or 
\item $0< a_i$ and $b_i = a_i +1$, or
\item there exists a sequence of indices $1 \leq j_1< \cdots < j_s < i$ such that the set 
$\{a_{j_1},\dots,a_{j_s}\}$ is equal to 
$\{a_i+1,\dots,a_i+s	\}$, and $b_i=a_i+s+1$.
\end{enumerate}
\end{Lemma}

\begin{Example}
Let $x=(4,0,5,0,3,1)$, and let $y=(4,0,5,0,6,1)$. Then $\ell(x)= 21$, and $\ell(y)=22$. 
If $z=(4,0,5,0,3,2)$, then $\ell(z)=22$.
\end{Example}

\begin{Lemma} \label{L:PPRcovering1}
Let $x=(a_1,\dots,a_n)$ and $y = (b_1,\dots,b_n)$ be two elements of $R_n$. 
Suppose that $a_j=b_i,\ a_i = b_j$ and $b_j < b_i$ where $i < j$. 
Furthermore, suppose that for all $k\in \{1,\dots\widehat{i},\dots,\widehat{j},\dots,n\}$, $a_k=b_k$.
Then, $\ell(y) = \ell( x)+ 1$ if and only if for $s=i+1,\dots,j-1$, either $a_j< a_s$, or $a_s < a_i$.
\end{Lemma}

\begin{Example}
Let $x=(2,6,5,0,4,1,7)$, and let $y=(4,6,5,0,2,1,7)$. 
Then $\ell(x)= 35$, and $\ell(y)=36$. Let $z=(7,6,5,0,4,1,2)$. Then $\ell(z)=42$.
\end{Example}

\section{Proof of Theorem~\ref{T:Borelmonoid}}~\label{S:proofof2}

Recall our notation that $B_n^{nil}$ denotes the strictly upper triangular 
elements of the Borel-Renner monoid $B_n$. 
Clearly, $B_n^{nil}$ is isomorphic to $B_{n-1}$
not only as a monoid but also as a poset. 

Let $\sigma$ be an element from $B_n^{nil}$ and let 
$\sigma_1\dots \sigma_n$ be its one-line notation. 
We associate an arc-diagram $A = A(\sigma)$ to $\sigma$ as follows. 
If $i$ and $j$ are two 
positive integers such that $1\leq i < j \leq n$, 
then $j$ covers $i$ in a chain of $A$ if and only if 
$\sigma_j = i$. Obviously, in this case, $\{i, j\}$ 
is an arc of $A$.  
This association is a version of a well-known bijection
between set partitions $\Pi_n$ and the rook placements
on an upper triangular board of base length $n-1$. 
Let us denote by $\varphi$ the bijection that is defined in the 
previous paragraph. Our goal in this section is 
to prove that 
\begin{align}\label{A:varphi}
\varphi : (\mc{A}_n,\prec) \rightarrow (B_n^{nil},\leq )
\end{align}
is a poset isomorphism.

Let $x=(a_1,\dots,a_n)$ and $y = (b_1,\dots,b_n)$ be elements of $B_n^{nil}$
such that $y$ covers $x$. By Theorem~\ref{T:PPR} we know that either 
\begin{enumerate}
\item there exists an $1 \leq i \leq n$ such that $b_i> a_i$ and $b_j = a_j$ for all $j\neq i$, or
\item  there exist  $1 \leq i < j \leq n$ such that $b_i=a_j,\ b_j=a_i$ with $b_i > b_j$, 
and for all $k\notin \{i,j\}$, $b_k = a_k$.
\end{enumerate}
Let us proceed with the first case. Then by Lemma~\ref{L:PPRcovering0}
we know that exactly one of the following statements hold true: 
\begin{enumerate}
\item[1.a] $0=a_i$, $b_i=1$ and $a_j=b_j > 0$ for all $j > i$, or 
\item[1.b] $0< a_i$ and $b_i = a_i +1$, or
\item[1.c] there exists a sequence of indices $1 \leq j_1< \cdots < j_s < i$ 
such that the set $\{a_{j_1},\dots,a_{j_s}\}$ is equal to 
$\{a_i+1,\dots,a_i+s	\}$, and $b_i=a_i+s+1$.
\end{enumerate}

In the case of 1.a we see that $\varphi^{-1}(x)$ 
has 1 as an isolated vertex and $\varphi^{-1}(y)$ has $\{1,j\}$
as an arc. Notice that no arc whose starting vertex 
is 1 lies under another arc. Moreover, since 
$a_j=b_j >0$ for all
$j>i$ by our hypothesis, the vertex $i$ has the biggest
possible index that the arc starting at 1 can connect.
Therefore, according to Rule 3. we have a covering
relation $\varphi^{-1}(x)\prec \varphi^{-1}(y)$.

In the case of 1.b, $\{a_i,i\}$ is an arc in $\varphi^{-1}(x)$
and in $\varphi^{-1}(y)$ we have $\{a_i+1,i\}$ as an arc. 
Therefore, an arc of $\varphi^{-1}(x)$ is shortened by 1,
hence according to Rule 1. this is a covering relation.

The case of 1.c is similar to 1.a; it gives a covering 
relation by Rule 3.

Next, we look at the second type of covering 
relation as in Theorem~\ref{T:PPR}. In this case, 
we look at the numbers $a_i,a_j,i$ and $j$ closely. 
By definition $\{a_i,i\}$ and $\{a_j,j\}$ 
are arcs in $\varphi^{-1}(x)$. But 
both of the arcs $\{a_j,i\}$ and $\{a_i,j\}$ are 
contained in $y$, therefore, 
$a_i< a_j < i < j$. This means that 
$\{a_i,i\}$ and $\{a_j,j\}$ are crossing arcs in $\varphi^{-1}(x)$.
However, the arcs $\{a_i,j\}$ and $\{a_j,i\}$ are nested
in $\varphi^{-1}(y)$. 
By Rule 2., we see that $\varphi^{-1}(y)$ covers $\varphi^{-1}(x)$.

In summary, we showed that the map $\varphi^{-1}$ is
an order preserving bijection from $(B_n^{nil},\leq)$
to $(\mc{A}_n,\prec)$.

Next, we will show that $\varphi$ is an order preserving 
bijection from $(\mc{A}_n,\prec)$ to $(B_n^{nil},\leq)$. 
Let $A$ and $B$ be two arc-diagrams such that 
$A$ is covered by $B$ in $(\mc{A}_n,\prec)$. 
Let $x$ and $y$ denote, respectively, the images
of $A$ and $B$ in $B_n^{nil}$. (We continue to use 
the one-line notation for the elements of $R_n$.)
If the covering relation $A\prec B$ is obtained from Rule 3., then
$y$ is obtained from $x$ by inserting a nonzero entry to $x$. 
But according to item 1. in Theorem~\ref{T:PPR}, 
this is a covering relation in $R_n$.
If the covering relation $A\prec B$ is obtained from Rule 2., then
item 2. in Theorem~\ref{T:PPR} applies.  
Finally, if the covering relation $A\prec B$ is obtained from Rule 1., then
there are two possibilities. To describe, let $a$ denote denote the arc
$a=\{v_i,v_j\}$ in $A$ such that to obtain $B$ from $A$ we 
replace exactly one of the vertices $v_i$ or $v_j$ by another vertex $v_k$
($i<k<j$). 
In the first possible scenario, $v_i$ is replaced by $v_k$. 
This amounts to a covering relation as described in
items 2 or 3 of Lemma~\ref{L:PPRcovering0}.
In the second possible scenario, $v_j$ is replaced by $v_k$.
This amounts to the covering relation as in Lemma~\ref{L:PPRcovering1}. 
Therefore, $\varphi$ is order preserving as well, 
hence the proof of Theorem~\ref{T:Borelmonoid} is finished.

\begin{proof}[Proof of the first claim of Theorem~\ref{T:main1}]
The poset $(B_n^{nil},\leq)$ is the interval $[(0,\dots, 0), (0,1,\dots, n-1)]$ 
in $(B_n,\leq)$, which, in turn, is the interval $[(0,\dots, 0),(1,\dots, n)]$ in 
$(R_n,\leq)$. Therefore, by Theorem~\ref{T:Borelmonoid}, 
the poset $(\mc{A}_n,\prec)$ is bounded, graded, and EL-shellable. 
\end{proof}

\section{Statistics on arc-diagrams}\label{S:Stats}

In this section, to prove the second part
of Theorem~\ref{T:main1}, we will show
in Proposition~\ref{P:second part} that the function $\mathtt{t}$ 
defined in the introduction section agrees with the
length function on $B_n$. 
Then we will give another combinatorial reformulation of $\mathtt{t}$.

\begin{Proposition} \label{P:second part}
Let $x$ be a partial permutation of the form $x=(a_1,\dots, a_n)\in R_n$ 
and let $A_x$ be the arc-diagram on $n$ vertices which corresponds to $x$. 
Then
\begin{align}\label{A:our claim}
\ell( x) =\mathtt{t}(A_x)\, .
\end{align}
\end{Proposition}

\begin{proof} 
We will use induction on the number of vertices, $n$. 
The base case of the induction is obvious. 
We assume that our claim (\ref{A:our claim}) is true for all arc-diagrams 
with at most $n$ vertices. 
We proceed to prove our claim for $n+1$. 
Let $x=(a_1,\dots,a_{n+1})\in R_{n+1}$ be a partial permutation 
with the corresponding arc-diagram $A_x$ on $n+1$ vertices and 
with $k$ arcs.

Let $s$ be a number such that $2\leq s \leq n+1$. 
There are two cases to consider.
First, if there exists an arc $\{1,s\}$ in $A_x$, then 
let $\tilde A$ denote the arc-diagram that is obtained from $A_x$ 
by removing $\{1,s\}$, and let $\tilde x$ be the partial permutation which 
corresponds to $\tilde A$. Then $k$ is the number of arcs in $\tilde A$. 
Clearly, although $\tilde A$ has $n+1$ vertices, since its first vertex 
does not have any arcs emanating from it, $\tilde x$ has a 0 in its first entry.
Removing this entry from $\tilde x$ and removing the first vertex from $\tilde A$
does not alter the difference between the length $\ell$ and the statistics $\mathtt{t}$. 
Indeed, if $\tilde x'$ and $\tilde A'$ denotes the resulting partial permutation
and the corresponding arc-diagram, then we see that  
$$
\ell ( \tilde x ) = \ell (\tilde x') +k \ \text{ and }\ \mathtt{t}(A_{\tilde x}) = \mathtt{t}(\tilde A') + k.
$$
Now by the induction
hypothesis, we have $\ell(\tilde x) =\mathtt{t}(\tilde A)$. 
Secondly, if there is no arc of the form $\{1,s\}$ in $A_x$,
then we repeat the previous argument.

Let us denote the $i$-th coordinates of $x$ and $\tilde x$ by 
$a_i(x)$ and $a_i(\tilde x)$, respectively, for $i=1,\dots, n+1$. 
Notice that the $s$-th entry of $\tilde x$ is 0, and the $s$-th entry of $x$ 
is 1. All other entries of $\tilde x$ and $x$ coincide. So, we have 
$$
\sum a_i (x)=\sum a_i(\tilde x)+1. 
$$
Notice also that $a_j(x)=0$ if and only if 
$j$ is the starting point of a chain in $A_x$, 
therefore, 
\begin{align}\label{A:r is as before}
inv(x)=inv(\tilde x)+1+r,
\end{align}
where $r$ is the number of chains in $A_x$ that 
start at the $j$-th vertex with $j>s$. 
Thus, $\ell(x)=\ell(\tilde x)+1+r$.

Next, we compare $\mathtt{t}(A_x)$ and $\mathtt{t}(\tilde A)$. 
We have
\begin{align}\label{A:meanings}
\mathtt{t}(A)=\mathtt{t}(\tilde A)+n-k-1-(s-2)+q\,,
\end{align}
where $q$ is the number of arcs under the arc $\{1,s\}$. 
Let us explain the meanings of the summands on the right side of (\ref{A:meanings}).
The summand $n-k-1$ appears since $A$ has one more 
arc than that $\tilde A$ has;  
the contribution of the arcs in $\tilde A$ to $\mathtt{t}(\tilde A)$ is 
$\sum_{i=1}^k (n-i)$ in $\mathtt{t}(\tilde A)$ whereas the contribution 
of arcs of $A_x$ to $\mathtt{t}(A_x)$ is $\sum_{i=1}^{k+1} (n-i)$ in $\mathtt{t}(A)$.
The summand $-(s-2)$ appears since the depths of each of the $s-2$ vertices 
$v_2$, $v_3$, ... , $v_{s-1}$ of $\tilde A$ increase by 1 when we include the arc $\{1,s\}$. 
Finally, the summand $q$ appears since the depths of each of the $q$ arcs on 
the vertices $v_2,v_2,\dots, v_{s-1}$ of $\tilde A$ increase by 1 when we add include the $\{1,s\}$.

Thus, in order to prove the equality $\ell( x) =\mathtt{t}(A)$ 
it suffices to show that $r=n-k-s+q$, where $r$ is as in (\ref{A:r is as before}).
This equality holds in view of the following argument; 
$n-k-r$ is the number of chains in $\tilde A$ starting at a verticex $v_l$ with $l\leq s$. 
If we add to this number the number of arcs on the vertices $v_2,v_3,\dots, v_{s-1}$, 
we get exactly $s$. 
Consider the truncated sub-diagram of $\tilde A$ on the first $s$ vertices.
(The arcs $\{ i,j\}$ with $i<s<j$ are deleted from $\tilde A$.) 
It is easy to see that, in any arc-diagram on $s$ vertices, the number of arcs 
plus the number of chains equals to $s$. 
Therefore the number of arcs in the truncated diagram is $q$, 
and the number of chains therein is $n-k-r$. 
So, $s=n-k-r+q$, or, $r=n-k-s+q$ is true.
This finishes the proof of the equality $\ell( x) =\mathtt{t}(A)$.
\end{proof}

Following the conventions that are set 
before Definition~\ref{depth} on the crossings 
of arcs, we define the ``crossing number'' of an arc
as follows. 

\begin{Definition}\label{cross}
Let $\alpha$ be an arc in an arc-diagram $A$. 
We denote by $cross(\alpha)$ the total number of chains that
$\alpha$ crosses. 
Note that $\alpha$ crosses a chain at most twice. In this case,  
we consider it as a single crossing.
\end{Definition}

\begin{Example}\label{E:1}

Let $A$ be the arc-diagram in Figure~\ref{F:crossings}. 
The crossing numbers of $A$ are as follows:
$cross(\{1,8\})=cross(\{2,5\})=cross(\{3,7\})=1$, $cross(\{5,6\})=0$, 
and $cross(\{6,9\})=2$. 
\begin{figure}[h]
\begin{center}
\begin{tikzpicture}[scale=.45]
\node at (-8,0) {$\bullet$};
\node at (-6,0) {$\bullet$};
\node at (-4,0) {$\bullet$};
\node at (-2,0) {$\bullet$};
\node at (0,0) {$\bullet$};
\node at (2,0) {$\bullet$};
\node at (4,0) {$\bullet$};
\node at (6,0) {$\bullet$};
\node at (8,0) {$\bullet$};

\node at (-8,-0.5) {$1$};
\node at (-6,-0.5) {$2$};
\node at (-4,-0.5) {$3$};
\node at (-2,-0.5) {$4$};
\node at (0,-0.5) {$5$};
\node at (2,-0.5) {$6$};
\node at (4,-0.5) {$7$};
\node at (6,-0.5) {$8$};
\node at (8,-0.5) {$9$};

\draw[thick,blue,-] (-8,0) ..controls  (-0.5,4) .. (6,0);
\draw[thick,blue,-] (-6,0) ..controls  (-3,1) .. (0,0);
\draw[thick,blue,-] (-4,0) ..controls  (0,2) .. (4,0);
\draw[thick,blue,-] (0,0) ..controls  (1,1) .. (2,0);
\draw[thick,blue,-] (2,0) ..controls  (5,1) .. (8,0);
\end{tikzpicture}
\end{center}
\caption{Crossings.}
\label{F:crossings}
\end{figure}
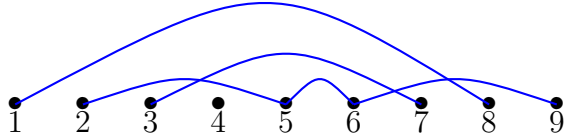
\end{Example}

The following proposition shows the relation between {\it cross} and {\it depth}.
\begin{Proposition}\label{crossdepth} 
Let $A$ be an arc-diagram on $n$ vertices, denoted by $v_1,\dots, v_n$.
Let $\alpha_1,\dots, \alpha_k$ denote its arcs, and let $\beta_1,\dots, \beta_{n-k}$
denote its chains. In this notation, the following equality holds true:
\begin{align}\label{A:crossdepth}
\sum_{m=1}^k cross(\alpha_m)=\sum_{i=1}^{n} 
depth(v_i)-\sum_{m=1}^{k} depth(\alpha_m)-\sum_{j=1}^{n-k} depth(\beta_j)\,.
\end{align}
\end{Proposition}

\begin{proof}
Once again, we use induction on $n$.  
The base case is obvious, so we assume that our claim (\ref{A:crossdepth})
holds true for arc-diagrams on $m \leq n-1$ vertices and we will prove it 
for the arc-diagrams on $n$ vertices. 

Now, let $A$ be an arc-diagram whose vertices, arcs, and chains are 
as in the hypothesis of the proposition. 
If there is no arc that emanates from the first vertex, then 
removal of the vertex does not alter neither the left hand side 
nor the right hans side of eqn. (\ref{A:crossdepth}). So, 
in this case, by the induction hypothesis, we see that (\ref{A:crossdepth})
holds true. Next, we will analyze
how both sides of eqn. (\ref{A:crossdepth}) changes if we add an arc $\{1,t\}$ 
to $A$. There are two cases.
We abbreviate ``left hand side of (\ref{A:crossdepth})'' to ``l.h.s.'' 
and similarly we abbreviate ``right hand side of (\ref{A:crossdepth})'' to ``r.h.s.''.

{\bf Case 1.} We assume that there is no arc of $A$ which is of the form $\{t,r\}$ with $r>t$. 
In this case, let us denote by $p$ the number of arcs which cross $\{1,t\}$.
In other words, the number of arcs $\{a,b\}$ such that $a<t$ and $b>t$ is $p$. 
If we add $\{1,t\}$ back to $A$, then $\sum cross(\alpha_k)$ increases by $2p$. 
Let us look at the r.h.s. The sum of depths of vertices, $\sum depth(v_i)$, 
increases by $t-2$ since $v_2,\dots,v_{t-1}$ are now below the arc $\{1,t\}$. 
The sum of the depths of arcs,  $\sum  depth(\alpha_m)$ increases by $q$, 
where $q$ is the number of arcs under $\{1,t\}$.
(In other words, $q$ is the number of arcs $\{i,j\}$ such that $2\leq i < j \leq t-1$.)
By adding $\{1,t\}$, we see that the sum of depths of chains, $\sum  depth(\beta_j)$ increases by $s$, 
where $s$ is the number of chains under the arc $\{1,t\}$; but also it decreases by $p$, 
since in $A$ the vertex $v_t$ was a chain by itself and there were $p$ arcs above it.
In conclusion, the l.h.s. increases by $2p$, while the r.h.s. increases by $t-2-q-s+p$. 
Notice also the equality $p+q+s=t-2$ which follows from the fact that 
in any arc-diagram on $n$ vertices the number of arcs plus the number of chains equals to $n$. 
Now, since $p+q+s=t-2$ is true, the r.h.s. and the l.h.s. are still equal after the arc $\{1,t\}$ is added to $A$. 
This finishes the proof of the first case.

{\bf Case 2.} Assume that $A$ has an arc of the form $\{t,r\}$ with $r> t$. 
Then, by adding $\{1,t\}$ to $A$, the l.h.s. increases by $2p-u$, 
where $p$ is the total number of arcs of $A$ that are of the form $\{a,b\}$ with $a<t$ and $b>t$, 
and $u$ is the number of arcs $\{a,b\}$ with $a<t$ and $b>t$ 
that cross the chain $\{t,r,\dots\}$ in $A$. 
Let us look at the r.h.s..
As in Case 1., the sum of depths of vertices, $\sum depth(i)$, increases by $t-2$; 
the sum of the depths of arcs, $\sum  depth(\alpha_m)$ increases by $q$, 
where $q$ is the number of arcs under the arc $\{1,t\}$. Finally, 
the sum of the depths of chains, $\sum  depth(\beta_j)$, changes as follows: 
it increases by $s$, where $s$ is the number of chains under the arc $\{1,t\}$, 
and it decreases by $p-u$, where $p$ and $u$ are as before.
In summary, the l.h.s increases by $2p-u$, while the r.h.s. increases by $t-2-q-s+p-u$.
Therefore, the l.h.s. and the r.h.s are equal in view of the equality $p+q+s=t-2$
which is seen as in Case 1.. This finishes the proof of our claim. 

\end{proof}

\begin{Definition}\label{sdc}
Let $A$ be an arc-diagram on $n$ vertices with $k$ arcs denoted 
$\alpha_1$, $\alpha_2$,...,$\alpha_k$ and $n-k$ chains denoted 
$\beta_1$, $\beta_2$,...,$\beta_{n-k}$.  We define the crossing-index 
of $A$ by the formula 
$$
\mathtt{c}(A)=\sum_{i=1}^k (n-i)-\sum_{j=1}^{n-k} depth(\beta_j)-\sum_{m=1}^k cross(\alpha_m)\,.
$$
\end{Definition}

\begin{Example}

We continue with Example~\ref{E:1}.
The arc-diagram $A$ consists of four chains, 
$\{1,8\}$, $\{2,5,6,9\}$, $\{3,7\}$, $\{4\}$; 
it has five arcs, $\{1,8\}$, $\{2,5\}$, $\{3,7\}$, $\{5,6\}$, $\{6,9\}$.
The depths of the vertices are $depth(1)=depth(9)=0$, $depth(2)=depth(8)=1$, 
$depth(3)=depth(5)=depth(6)=depth(7)=2$, and $depth(4)=3$. 
The depths of arcs are given by 
$depth(\{ 1,8\} )=depth(\{ 6,9\})=0$, $depth( \{2,5\} )=depth(\{ 3,7\})=1$, $depth(\{5,6\})=2$. 
Therefore, the depth-index of $A$ is given by 
$$
\mathtt{t}(A)=8+7+6+5+4-(0+1+2+3+2+2+2+1+0)+(0+1+1+2+0)=21\,.
$$
Next, we will compute the crossing-index $\mathtt{c}(A)$. 
The depths of chains are given by $depth(\{1,8\}=depth(\{2,5,6,9\})=0$,
$depth(\{3,7\})=1$, $depth(\{4\})=3$. 
The crossing numbers of $A$ are as follows: 
$cross(1,8)=cross(2,5)=cross(3,7)=1$, $cross(5,6)=0$, $cross(6,9)=2$. In summary we have 
$$
\mathtt{c}(A)=8+7+6+5+4-(0+0+1+3)-(1+1+1+0+2)=21\,.
$$
\end{Example}

The equality of the depth-index and the crossing-index holds true for all arc-diagrams. 
The proof of this fact follows from Proposition~\ref{crossdepth} and the definitions, 
so we omit it. 
\begin{Proposition} 
Let $A$ be an arc-diagram. Then
$$
\mathtt{t}(A)=\mathtt{c}(A)\,.
$$
\end{Proposition}

\section{Stirling posets}\label{S:Subposets}

Recall that $E_{n,k}$ denotes the set of diagaonal 
matrices $A$ such that $A$ is a diagonal matrix 
of rank $k$ and it has only 0's and 1's in its entries. 
Recall also that we defined the variety $\mt{B}_{n,k}$ 
as the union $\bigcup_{e\in E_{n,k}} \overline{\mt{B}_n e \mt{B}_n}$. 
In this section we will further explain and prove the 6 properties 
about $\mt{B}_{n,k}$'s that we listed in Introduction.

We start with proving the following lemma. 
\begin{Lemma}
that the number of 
components of $\mt{B}_{n,k}$ is ${n \choose k}$ and they are all equal dimensional.
\end{Lemma}
\begin{proof}

We will show that the elements of $E_{n,k}$ are incomparable 
in BCR ordering and furthermore 
\begin{align}\label{A:dimension of e}
e\in E_{n,k} \implies \ell(e) = \frac{k (2n-k+1)}{2}.
\end{align}
To this end, let $e$ and $f$ be two diagonal idempotents from $B_n$. 
By Remark~\ref{R:Renner=idempotent order} we know that 
$$
\mt{B}_{n} e \mt{B}_{n} \subset \overline{\mt{B}_{n} f \mt{B}_{n}} \iff ef= e.
$$
But for two diagonal matrices $e$ and $f$ with 0/1 entries and 
which are of the same rank, the equality $e=fe$ holds true if and only if $e=f$.  
Since each Borel orbit closure is an irreducible variety, 
it follows that each closed subset $\overline{\mt{B}_{n} f \mt{B}_{n}}$ ($f\in E_{n,k}$)
of $\mt{B}_{n,k}$ is irreducible, and these are precisely the irreducible components
of $\mt{B}_{n,k}$.
In particular, there are ${n \choose k}$ of them.

Next, we prove the length formula (\ref{A:dimension of e}).
We will accomplish this by inducting on $k$. We start with the 
base case $k=1$. 
Let 
$$
e= (a_1,a_2,\dots, a_n) \in E_{n,k}
$$
be an idempotent that is given in one-line notation (\ref{E:oneline}). 
Since $k=1$, there exists a unique index $i$ ($1\leq i \leq n$) such that 
$a_i=1$ and $a_j=0$ if $j\neq i$. Then by Proposition~\ref{P: sum inv} we know that
$\ell(e) = \sum a_i + inv(e) = i + (n-i) = n$ which agrees with (\ref{A:dimension of e}).
Now assume that our claim holds true for all idempotents of rank $k-1$, 
we proceed to show that it is true for $e= (a_1,a_2,\dots, a_n) \in E_{n,k}$. 
Let $m$ denote the largest index such that $a_m =m$. 
In this case, replacing this 1 with 0 gives us an element 
$$
e' = (a_1,\dots, a_{m-1},0,a_{m+1},\dots, a_n) \in E_{n,k-1}
$$ 
hence $\ell(e') = \frac{ (k-1)(2n-k+2)}{2}$ by our induction assumption. 
By Proposition~\ref{P: sum inv}, we see that 
\begin{align*}
\ell( e') &= \sum_{j\neq m} a_j + inv( e') \\
&= \left( \sum_{j} a_j +m\right) + \left( inv( e) - (k-1) + (n-m) \right) \\
&= \left( \sum_{j} a_j + inv( e)\right)  + ( m- (k-1) + (n-m)) \\
&= \ell( e) + (n-k+1)
\end{align*}
from which our claim follows.  Hence, the proof is finished. 
\end{proof}

\begin{Remark}\label{R:simple but useful}
A simple but useful fact regarding Bruhat-Renner-Chevalley order
on $R_n$ is that if $\sigma \leq \tau$ for two elements $\sigma, \tau \in R_n$,
then their matrix ranks satisfy $rank(\sigma )\leq rank(\tau)$. 
\end{Remark}

\begin{Lemma}\label{L:RennerStirlingdecomposition1}
For every $k$ in the range $0,\dots, n$, 
$\mt{B}_{n,k}$ has a Renner decomposition of the form 
\begin{align*}
\mt{B}_{n,k}=  \bigsqcup_{\sigma \in B_{n,k}} \mt{B}_n \sigma \mt{B}_n,
\end{align*}
where $B_{n,k}$ is a finite subsemigroup of $B_n$ and it consists of 
rooks from $B_n$ whose matrix rank is at most $k$. 
\end{Lemma}
\begin{proof}
This follows from the definition of Bruhat-Chevalley Renner ordering
and Remark~\ref{R:simple but useful}.
\end{proof}

\begin{Lemma}\label{L:filter}
If $n$ is a nonnegative integer $n$, then there is a filtration 
\begin{align*}
\{ \mathbf{0}\} = \mt{B}_{n,0} \subset \mt{B}_{n,1} \subset \cdots \subset \mt{B}_{n,n-1} \subset 
\mt{B}_{n,n} = \BM{n}.
\end{align*}
\end{Lemma}

\begin{proof}
This from the fact that 
$$
\overline{\mt{B}_n e \mt{B}_n } = \cup_{f\leq e} \mt{B}_n f \mt{B}_n \ 
\text{ and that } \ f \leq e \Rightarrow rank(f) \leq rank(e),
$$
where $e$ and $f$ are from $R_n$. 
\end{proof}

\begin{Corollary}
If $n$ and $k$ are two nonnegative integers such that 
$0\leq k \leq n$, then $\mt{B}_{n,k}$, hence $B_{n,k}$ (for all $k=0,\dots, n$) 
have the structure of an algebraic semigroup. 
\end{Corollary}
\begin{proof}
Since the rank of the product of two matrices is bounded by the 
ranks of the multiplicands, our claim follows from Lemma~\ref{L:filter}.
\end{proof}

\begin{Lemma}
The subsemigroups $B_{n,k}\subset B_n$ form a 
flag $\{\mathbf{0}\}\subset B_{n,1}\subset \cdots \subset B_{n,n}=B_n$
and moreover the number of elements of the difference $B_{n,k}-B_{n,k-1}$ is 
given by the Stirling number $S(n+1,n+1-k)$. 
\end{Lemma}

\begin{proof}
The first claim follows from Lemmas~\ref{L:RennerStirlingdecomposition1}
and~\ref{L:filter}. The second claim follows from the fact that 
$B_{n,k}-B_{n,k-1}$ consists of upper triangular rooks whose
matrix rank is $k$ and that the set partitions of $\{1,\dots, n+1\}$
with $k$ blocks is in bijection with upper triangular $n\times n$
rook matrices of rank $k$. 
\end{proof}

\begin{Proposition}\label{P:Pnk}
The Bruhat-Chevalley-Renner ordering
restricted to the subsets of the form $B_{n,k}-B_{n,k-1}$
(for $k=1,\dots, n$) is a graded poset 
with a minimum element. It has ${n \choose k}$ maximal elements. 
Each maximal interval in this poset is an interval in $B_n$,
therefore, it is an EL-shellable poset. 
\end{Proposition}

\begin{proof}
Let $P_{n,k}$ denote $(B_{n,k}-B_{n,k-1},\leq)$, where $k\in \{1,\dots, n\}$.
The poset $(P_{n,k},\leq)$, where $\leq$ is the Bruhat-Chevalley-Renner
ordering has ${n \choose k}$ maximal elements which are given by
the incomparable diagonal $n\times n$ rook matrices of rank $k$. 
It is easy to check that the rook matrix (given in one-line notation) 
$$
e_0:= (0,\dots, 0,1,2,\dots, k) \in B_n
$$
is the smallest element of $(P_{n,k},\leq)$.
Therefore, $P_{n,k}$ is a union of ${n\choose k}$ maximal intervals 
all of which has the same poset rank. It is clear that these maximal
intervals are intervals in $B_n$, hence in $R_n$ as well. In particular
we see that $P_{n,k}$ is an EL-shellable poset. 
\end{proof}

Now we are ready to prove Theorem~\ref{T:Borelmonoid2}
which states that for every pair of nonnegative 
integers $(n,k)$ such that $0\leq k \leq n$ 
the Stirling poset $\mc{A}_{n,k}$ 
is a graded and EL-shellable poset. Furthermore,
the cardinality of $\mc{A}_{n,k}$ is given by $S(n, n-k)$.

\begin{proof}[Proof of Theorem~\ref{T:Borelmonoid2}]
The second claim of the theorem is straightforward
to prove. To prove the first claim, 
following the notation in the proof of Proposition~\ref{P:Pnk}, 
we denote the poset $(B_{n,k}-B_{n,k-1},\leq)$ by $P_{n,k}$. 
Recall from Section~\ref{S:proofof2} that there is 
a poset isomorphism $\varphi : (\mc{A}_n,\prec) \rightarrow (B_{n-1},\leq )$
that is defined by $\varphi(A)=\sigma$ whenever $A$ and $\sigma$ are 
related as follows: $\{i,j\}$ is an arc in $A$ if and only if $\sigma_j=i$. 
In particular, if $A$ is an arc-diagram 
with $k$ chains, then $\sigma$ is a rook matrix of rank $n-k$.
Therefore, we see that $\mc{A}_{n,k}$ is isomorphic 
to $P_{n-1, n-k}$, hence $\mc{A}_{n,k}$ is graded and EL-shellable.
\end{proof}

\begin{Convention}\label{C:AP}
In the light of Theorem~\ref{T:Borelmonoid2}, whenever it is more convenient, 
we will use the poset $(B_{n-1,n-k}-B_{n-1,n-k+1},\leq)$, which is abbreviated 
to $P_{n-1,n-k}$, in place of $(\mc{A}_{n,k},\prec)$. 
\end{Convention}
\vspace{.5cm}

We proceed to prove Theorem~\ref{T:(n,2) is boolean} which states that 
$\mc{A}_{n,2}$ is the boolean lattice $B(n-1)-\{\{1,\dots, n-1\}\}$. 

\begin{proof}[Proof of Theorem~\ref{T:(n,2) is boolean}]
Following Convention~\ref{C:AP}, we will identify $P_{n-1,n-2}$ with $(\mc{A}_{n,2},\prec)$. 
A rook matrix $x\in B_n$ is an element of $P_{n-1,n-2}$ if its one-line notation 
$x= (a_1,\dots, a_{n-1})$ satisfies the following properties:
\begin{itemize}
\item the cardinality of $\{a_1,\dots, a_{n-1}\}$ is $n-1$. This means that the entries of $x$ are mutual different.
\item For $i=1,\dots, n-1$, $0\leq a_i \leq n-1$.
\end{itemize}
Clearly, the smallest element of $P_{n-1,n-2}$ is 
$x_0=(0,1,2,\dots, n-2)$. Indeed, $x_0$ has no inversions, and the sum of entries
of $x_0$ is the unique minimum of the function 
$$
x=(a_1,\dots, a_{n-1}) \longmapsto \sum_{i=1}^{n-1} a_i
$$ 
on $P_{n-1,n-2}$.

Next, we define the map $\psi: P_{n-1,n-2}\longrightarrow B(n-1)-\{\{1,\dots, n-1\}\}$ by 
$$
\psi(x)  = \{ a_i :\ a_i =i, \text{ where }  i\in \{1,\dots, n-1\} \}.
$$
Our goal is to prove that $\psi$ is a poset isomorphism by showing that for every  
$x=(a_1,\dots, a_{n-1})$ from $P_{n-1,n-2}$, the interval $[x_0,x]$ is isomorphic to $B(r)$, 
where $r$ is the number of indices $i=1,\dots, n-1$ such that $a_i =i$. 
To this end, we proceed by induction on $n$. The base case is when $n=2$ and 
in this case $P_{2,1}$ is isomorphic to a fish net poset with 3 elements, so our claim holds true. 
(In a similar manner, $P_{3,2}$ case be checked by hand.)

Now, let $x=(a_1,\dots, a_{n-1})$ be an element with $r$ fixed points, that is to say 
the cardinality of $\{a_i : \ a_i = i \text{ where } i\in \{ 0,1,\dots, n-1\} \}$ is $r$. 
We notice that the non-fixed entries of $x$ appear in an increasing order. 
In other words, if $a_{i_1},\dots, a_{i_{n-1-r}}$ are the entries in $x$ such that 
$0\leq a_{i_j} < i_j$ ($j=1,\dots, n-1-r$), then $a_{i_1} < \cdots < a_{i_{n-1-r}}$. 
Next, we observe that if $x$ covers $y$ in $P_{n-1,n-2}$, then 
$y$ is obtained from $x$ by interchanging exactly one of the fixed entries $a_i$
($=i$) with a non-fixed entry $a_j$ ($\neq j$). In this case, by our induction 
hypothesis $[x_0,y]$ is isomorphic to the Boolean lattice $B(r-1)$. 
Since there are exactly $r$ such subintervals in $[x_0,x]$, we see that $[x_0,x]$
is isomorphic to $B(r)$, hence the proof is finished.

\end{proof}

\section{Intervals in $\mc{A}_n$}\label{S:Intervals}

From an algebraic point of view, the Borel submonoid $\BM{n}$ 
may look much simpler compared to its ambient monoid $\mt{Mat}_n$.
Our goal in this section is to show that, once its size is doubled, 
the Borel monoid $\BM{2n}$ packs at least the same amount of combinatorial information 
as $\mt{Mat}_n$ does. 

We start with proving Theorem~\ref{T:intervals}. Its first three items states the
following: 
\begin{enumerate}
\item The interval $([Y(n),X(n)],\prec)$ in $\mc{A}_{2n}$ is isomorphic to $(S_n,\leq)$. 
\item The interval $([Z(n),Y(n)],\prec)$ in $\mc{A}_{2n}$ and is isomorphic to $(B_n,\leq)$.
\item The interval $([Z(n),X(n)],\prec)$ in $\mc{A}_{2n}$ is isomorphic to $(R_n,\leq)$. 
\end{enumerate}
where $X(n),Y(n),Z(n)$ and $W(n)$ are as in Figure~\ref{F:allofthem}.
\begin{proof}[Proof of Theorem~\ref{T:intervals}]
For the proofs of these statements, once again, we will use the 
poset isomorphism (\ref{A:varphi}) between $\mc{A}_{2n}$ and $B_{2n-1}$. 
In particular, the one-line notation for the rook matrices 
$\varphi(W(n)),\varphi(Z(n)),\varphi(X(n))$, and $\varphi(Y(n))$ are given by 
\begin{align*}
\varphi ( W(n)) &= (0,1,2,\dots, n,0,\dots, 0), \\ 
\varphi ( Z(n)) &= (0,\dots, 0), \\ 
\varphi ( Y(n)) &= (0,\dots,0,1,2,\dots, n), \\ 
\varphi ( X(n)) &= (0,\dots,0,n,n-1,\dots, 1).
\end{align*}
It is easy to see, by using Theorem~\ref{T:PPR}, that
if two rooks $x= (a_1,\dots, a_{2n})$ and $y=(b_1,\dots,b_{2n})$
have their first $n$ entries the same, that is $a_i=b_i$ for $i=1,\dots, n$,
then $x\leq y$ if and only if $(a_{n+1},\dots, a_{2n})\leq (b_{n+1},\dots,b_{2n})$.
The proofs of our claims 1.,2., and 3. follow easily from this simple observation.

To explain and prove the last item, 
we briefly review ``Ding's Schubert varieties.''
Let $\mt{Mat}_{n,m}$ denote the set of all $n\times m$ matrices of rank $n$,
hence we implicitly assume that $m\geq n$.
Let $\lambda = (\lambda_1,\dots, \lambda_{n})$ be a partition with $\lambda_i \geq \lambda_{i+1}$
for $1\leq i \leq n-1$ and $\lambda_1 = m$. 
For us, a Ferrers board $F_\lambda$ is a top-right justified subarray in an $n\times m$ 
matrix such that the length of the $i$-th row is $\lambda_i$. For example, 
if $\lambda = (6,3,1)$ (hence $n=3,m=6$), then the corresponding Ferrers board is of the form 
\begin{align*}
\begin{bmatrix}
* & * & * & * & * & * \\
& &  &  * & * & * \\
& &  &  &   & * 
\end{bmatrix},
\end{align*}
We denote by $M_\lambda$ the set 
$M_\lambda = \{ (a_{i,j}) \in \mt{Mat}_{n,m}:\ a_{i,j}=0\ \text{ if } (i,j) \notin F_\lambda \}$.
As it is shown in~\cite{Ding}, the quotient space $\mt{B}_n \backslash M_\lambda$, 
which is denoted by $X_\lambda$ 
has the structure of a smooth projective variety and it is noticed by 
Develin, Martin, and Reiner that 
$X_\lambda$ is actually isomorphic to a 
Schubert variety $X_w$ in $\mt{GL}_m/\mt{P}_\lambda$,
where $\mt{P}_\lambda$ is the parabolic subgroup of matrices 
of the form 
\begin{equation*}
\begin{pmatrix}
A_1 & *  \\
0 &  A_2  
\end{pmatrix},
\end{equation*}
where $A_1$ is an upper triangular invertible $n\times n$ matrix, 
and $A_2$ is an $m-n\times m-n$ invertible matrix. 
(See Section 2 of~\cite{DMR}. Note that in the cited reference 
the authors use flags rather than matrices to describe the partial flag varieties
and their Schubert varieties.)

Now, we specialize to our situation by choosing 
$\lambda =( 2n-1, 2n-2,\dots, n-1)$.
Thus, we have $m=2n-1$. However, we will view $F_\lambda$ as a 
top-right justified subarray of a $2n\times 2n$ matrix, hence 
$M_\lambda$ is contained in $\mt{Mat}_{2n}$ as an affine subvariety. 
It is not difficult to check that $M_\lambda$ is closed under 
the two sided action of $\mt{B}_{2n}\times \mt{B}_{2n}$. 
Consider the following subgroup of $\mt{B}_{2n}$:
$$
H:=\left\{ \begin{bmatrix} A & \mbf{0} \\ \mbf{0} & id_n \end{bmatrix}:\ A\in \mt{B}_n \right\}.
$$
Clearly, $H$ is isomorphic to $\mt{B}_n$ and the left multiplication action
of $\mt{B}_{2n}$ on $M_\lambda$ is equivalent to the left multiplication action of $\mt{B}_n$.
Therefore, the isomorphism $\mt{B}_n \backslash M_\lambda \rightarrow X_\lambda$ 
is $\mt{B}_{2n}$-equivariant, where the $\mt{B}_{2n}$ action on $M_\lambda$ is on the right
and $\mt{B}_{2n}$ action on $X_\lambda$ is on the left. 

Now we need a general fact about the topology of Schubert varieties. 
Let $G$ be a reductive algebraic group, $P\subset G$ be a parabolic subgroup
containing a Borel subgroup $B$. 
It is a well known fact that for every Schubert variety $X$ in $G/P$,
the orbits of $B$ in $X$ are affine spaces
and furthermore these affine spaces give a cell decomposition. 
Therefore, in our case, the cells of $X_\lambda$ are given by the $\mt{B}_{2n}$ orbits
in $X_\lambda$. Said differently, each $\mt{B}_n\times \mt{B}_{2n}$ orbit in $M_\lambda$
corresponds to an affine cell in $X_\lambda$. Since these orbits
are the same as the orbits of $\mt{B}_{2n}\times \mt{B}_{2n}$, we see that 
the partial order $\leq$ restricted to the rook matrices in
$M_\lambda$ describes the inclusion poset on the cells of the Schubert variety $X_\lambda$. 
It is clear that $\varphi ( W(n)) = (0,1,2,\dots, n,0,\dots, 0)$ corresponds to the maximal 
dimensional cell and $\varphi ( Y(n)) = (0,\dots,0,1,2,\dots, n)$ corresponds to the smallest
dimensional cell, hence our proof is finished. 

\end{proof}

\section{A recurrence}\label{S:Recurrence}

We already pointed out in the introductory section that 
the generating function ${n \bb k}$ of $\mc{A}_{n,k}$ 
is a $q$-analog of the Stirling numbers of second kind. 
Another closely related $q$-analog, which is introduced by 
Garsia~\cite{Garsia} and studied in~\cite{Milne,GarsiaRemmel}
is as follows. 
For $k=0,\dots, n$, $S_{n,k}(q)$ is defined as the polynomial
that solves the recurrences
\begin{align}\label{A:GR}
S_{n+1,k} (q) = q^{k-1} S_{n,k-1}(q)  + [k]_q S_{n,k}(q) 
\end{align}
with initial conditions $S_{0,0}=1$ and conventions 
$S_{n,k}=0$ whenever $0>k$ or $k>n$.
It is shown by Garsia and Remmel in~\cite{GarsiaRemmel} that 
\begin{align}\label{A:GR S=R}
S_{n,k}(q)=R_{n-k}(\delta_n,q),
\end{align} 
where $R_{k}(\delta_n,q)$ (for $k=0,\dots, n$) is the combinatorially defined function 
\begin{align}\label{A:GR rook poly}
R_{k}(\delta_n,q) = \sum_{\textbf{r} \in C_k(\delta_{n})} q^{\text{stat}(\textbf{r})}. 
\end{align}
Here $\delta_n$ is the 
staircase board, namely the bottom-right justified 
arrangement of boxes with $i-1$ boxes in the $i$-th row, 
$C_k(\delta_{n})$ is the set of all placements of $k$ non-attacking rooks in $\delta_n$.
Clearly, $k$-rook placement can be thought of as an element of $P_{n,k}$ 
by turning the staircase board up-side-down and then completing it to a square 
$n\times n$ matrix with 0's and 1's where $1$'s represent the placements of non-attacking rooks.  
Finally, the statistics in (\ref{A:GR rook poly}) is the inversion statistics of the rook placements. 
Rather than defining this combinatorial statistic on rook placements, 
which is somewhat lengthy, we will 
mention a useful result that gives us an equivalent form.
The following observation is recorded in~\cite[Lemma 5.3]{CanRenner08} 
in a slightly different terminology.
\begin{Lemma}\label{L:CanRenner08}
Let $\sigma$ be a rook matrix from $P_{n-1,k}$ and 
let $\textbf{r}=\textbf{r}(\sigma)$ denote the corresponding 
non-attacking $k$-rook placement in $\delta_n$. In this case, 
the following equality holds true
$$
\dim ( \mt{B}_n \sigma \mt{B}_n) = {n \choose 2} - \text{stat}(\textbf{r}).
$$
\end{Lemma}
As a consequence of Lemma~\ref{L:CanRenner08} and definitions, 
\begin{align}\label{A:ours vs GR}
{n+1 \bb k}= q^{ {n+1\choose 2} } R_k \left(\delta_{n+1},\frac{1}{q}\right).
\end{align}

Now we are ready to finish our paper by proving 
Theorem~\ref{T:gen func}.

\begin{proof}[Proof of Theorem~\ref{T:gen func}]

By (\ref{A:GR S=R}) and the recurrence relation in (\ref{A:GR}), we have 
\begin{align}
q^{ {n+1\choose 2} } R_k \left(\delta_{n+1},\frac{1}{q}\right) 
&= q^{ {n+1\choose 2} }  S_{n+1,n+1-k} \left(\frac{1}{q}\right) \notag \\
&= q^{ {n+1\choose 2} }  q^{-(n-k)} S_{n,n-k}\left(\frac{1}{q}\right) + q^{ {n+1\choose 2} } [n+1-k]_{\frac{1}{q}} S_{n,n+1-k}\left(\frac{1}{q}\right) \notag \\
&= q^{ {n\choose 2} +k} R_k \left(\delta_{n},\frac{1}{q}\right) + q^{ {n+1\choose 2} -(n-k)} [n+1-k]_q 
R_{k-1} \left(\delta_{n},\frac{1}{q}\right). \label{A:son}
\end{align}
It follows from (\ref{A:ours vs GR}) and (\ref{A:son}) that 
\begin{align}\label{A:first part}
{n+1 \bb k} = q^{k} {n \bb k} + [n+1-k]_q q^{k} {n\bb k-1}.
\end{align}
This finishes the proof of our theorem.

\end{proof}

\end{document}